\documentclass[a4paper,12pt,twoside,onecolumn,final]{article}
\usepackage{amsmath,amsfonts,amssymb,amscd,mathrsfs}
\usepackage{extarrows}
\usepackage{epsfig}
\usepackage{enumerate}
\usepackage{color}
\usepackage{graphicx}
\usepackage{psfrag}
\usepackage[nodayofweek]{datetime}
\usepackage{bbm}
\usepackage{subcaption}
\usepackage[dvips,paper=letterpaper, inner=2.5cm,outer=2.3cm,top=2.0cm,bottom=2.5cm,includehead]{geometry}
\usepackage[ntheorem]{empheq} 
\empheqset{box=\bigfbox}
\usepackage[thmmarks,amsmath,hyperref]{ntheorem} 
\theoremstyle{plain} 
\theoremheaderfont{\normalfont\bfseries} 
\theorembodyfont{\slshape}
\theoremsymbol{\ensuremath{_\Box}}
\theoremnumbering{Alph}
\newtheorem{theorem}{Theorem}[section]

\theoremnumbering{arabic}
\newtheorem{proposition}{Proposition}[section]

\theoremstyle{plain}
\theoremheaderfont{\normalfont\bfseries} 
\theorembodyfont{\upshape}
\theoremsymbol{\ensuremath{_\Box}}
\newtheorem{remark}[proposition]{Remark}  

\theoremstyle{nonumberplain}
\theoremheaderfont{\normalfont\bfseries} 
\theorembodyfont{\upshape}
\theoremsymbol{\ensuremath{_\blacksquare}} 
\newtheorem{proof}{Proof}
\theoremsymbol{\ensuremath{_\blacksquare}}

\theoremlisttype{allname}

\usepackage[notref,notcite,draft]{showkeys}

\definecolor{refkey}{rgb}{0,1,0} 
\definecolor{labelkey}{rgb}{0,1,1}

\title{\textsc{Data-Driven Finite Elements Methods: Machine Learning Acceleration of Goal-Oriented Computations}\\
\textsc{}
\\ 
}
\author{I.~Brevis\footnotemark[1], I.~Muga\footnotemark[2], and K.G.~van~der~Zee\footnotemark[3]}
\date{\today}

\newlength{\bigfboxsep}
\newcommand{\bigfbox}[1]{\setlength{\fboxsep}{\bigfboxsep}\fbox{#1}}

\newcommand{\Span}{\operatorname{Span}}
\newcommand{\ANN}{\operatorname{ANN}}
\DeclareMathOperator*{\argmin}{argmin}

\newcommand{\xnot}{x_{0}}

\newcommand{\mtbb}[1]{\mathbb{#1}}

\newcommand{\mbbN}{\mtbb{N}}

\newcommand{\mbbR}{\mtbb{R}}

\newcommand{\mbbU}{\mtbb{U}}
\newcommand{\mbbV}{\mtbb{V}}

\usepackage{marginnote}
\begin{document}
%
%
\maketitle
%
%
\renewcommand{\thefootnote}{\fnsymbol{footnote}}
\footnotetext[1]{Pontificia Universidad Cat\'olica de Valpara\'iso, Instituto de Matem\'aticas.\\ \hspace*{15pt} ignacio.brevis.v@gmail.com}
\footnotetext[2]{Pontificia Universidad Cat\'olica de Valpara\'iso, Instituto de Matem\'aticas.\\ \hspace*{15pt} ignacio.muga@pucv.cl}
\footnotetext[3]{University of Nottingham, School of Mathematical Sciences.\\ \hspace*{15pt} kg.vanderzee@nottingham.ac.uk}
\renewcommand{\thefootnote}{\arabic{footnote}}
%
%
\begin{abstract}
\noindent We introduce the concept of data-driven finite element methods. These are finite-element discretizations of partial differential equations (PDEs) that resolve quantities of interest with striking accuracy, regardless of the underlying mesh size. The methods are obtained within a machine-learning framework during which the parameters defining the method are tuned against available training data. In particular, we use a stable parametric Petrov--Galerkin method that is equivalent to a minimal-residual formulation using a weighted norm. While the trial space is a standard finite element space, the test space has parameters that are tuned in an off-line stage. Finding the optimal test space therefore amounts to obtaining a goal-oriented discretization that is completely tailored towards the quantity of interest. 
As is natural in deep learning, we use an artificial neural network to define the parametric family of test spaces. Using numerical examples for the Laplacian and advection equation in one and two dimensions, we demonstrate that the
data-driven finite element method has superior approximation of quantities of interest even on very coarse meshes
%
%
%
\end{abstract}
%
%

%

%
\newpage
\tableofcontents
%
%
\pagestyle{myheadings}
\thispagestyle{plain}
\markboth{\small I. Brevis, I. Muga, and K.G. van der Zee}{\small Data-driven FEM: Machine learning acceleration of goal-oriented computations}
%
\section{Introduction}

%

In this paper we consider the data-driven acceleration of Galerkin-based discretizations, in particular the finite element method, for the approximation of partial differential equations (PDEs). The aim is to obtain approximations on meshes that are \emph{very} coarse, but nevertheless resolve quantities of interest with \emph{striking} accuracy. 
\par
We follow the machine-learning framework of Mishra~\cite{mishra2018machine}, who considered the data-driven acceleration of \emph{finite-difference} schemes for ordinary differential equations (ODEs) and PDEs. In Mishra's machine learning framework, one starts with a \emph{parametric family} of a stable and consistent numerical method on a fixed mesh (think of, for example, the $\theta$-method for ODEs). Then, a training set is prepared, typically by offline computations of the PDE subject to a varying set of data values (initial conditions, boundary conditions, etc), using a standard method on a (very) fine mesh. Accordingly, an optimal numerical method on the coarse grid is found amongst the general family, by minimizing a loss function consisting of the errors in quantities of interest with respect to the training data.
\par
The objective of this paper is to extend Mishra's machine-learning framework to \emph{finite element methods}. The main contribution of our work lies in the identification of a proper stable and consistent general family of finite element methods for a given mesh that allows for a robust optimization. In particular, we consider a parametric Petrov--Galerkin method, where the trial space is fixed on the given mesh, but the test space has trainable parameters that are to be determined in the offline training process. Finding this optimized test space therefore amounts to obtaining a coarse-mesh discretization that is completely tailored for the quantity of interest.
\par
A crucial aspect for the stability analysis is the equivalent formulation of the parametric Petrov--Galerkin method as a \emph{minimal-residual} formulation using discrete dual norms. Such techniques have been studied in the context of discontinuous Petrov--Galerkin (DPG) and optimal Petrov--Galerkin methods; see for example the overview by Demkowicz \& Gopalakrishnan~\cite{DemGopBOOK-CH2014} (and also~\cite{MugZeeARXIV2018} for the recent Banach-space extension). A key insight is that we can define a suitable test-space parametrization, by using a (discrete) trial-to-test operator for a test-space norm based on a parametric weight function. This allows us to prove the stability of the parametric minimal-residual method, and thus, by equivalence, proves stability for the parametric Petrov--Galerkin method. 
\par
As is natural in \emph{deep learning}, we furthermore propose to use an \emph{artificial neural network} for the weight function defining the test space in the Petrov--Galerkin method. The training of the tuning parameters in the neural network is thus achieved by a minimization of a loss function that is implicitly defined by the neural network (indeed via the weight function that defines the test space, which in turn defines the Petrov-Galerkin approximation, which in turn leads to a value for the quantity of interest). 
\par
\par
%
%
\par

\subsection{Motivating example}\label{sec:motivation}
To briefly illustrate our idea, let us consider a very simple motivating example. We consider the following simple 1-D elliptic boundary-value problem:
\begin{equation}
\label{eq:intro_main_equation}
\left\{
\begin{array}{ll}
-u_\lambda''= \delta_\lambda & \hbox{in } (0,1),\\
u_\lambda(0) =u_\lambda'(1)=0, 
\end{array}
\right.
\end{equation}
where $\delta_\lambda$ denotes the usual Dirac's delta distribution centered at the point 
$\lambda\in(0,1)$. The quantity of interest (QoI) is the value $u_\lambda(\xnot)$ of the solution at some fixed point $\xnot\in(0,1)$. 
\par
The standard variational formulation of problem~\eqref{eq:intro_main_equation} reads:
\begin{equation}
\label{eq:intro_model}
\left\{
\begin{array}{l}
\text{Find } u_\lambda\in H_{(0}^{1}(0,1) \text{ such that:} \\
\displaystyle \int_{0}^{1} u_\lambda'v' = v(\lambda), \qquad\forall v\in H_{(0}^{1}(0,1), 
\end{array}
\right.
\end{equation}
where $H_{(0}^{1}(0,1):= \{v\in L^2(0,1): v'\in L^2(0,1) \wedge v(0)=0\}$.
For the very coarse discrete subspace $\mbbU_h:=\Span\{\psi\}\subset H_{(0}^{1}(0,1)$ consisting of the single linear trial function~$\psi(x) = x$, the usual Galerkin method approximating~\eqref{eq:intro_model} delivers the discrete solution $u_h(x)=\lambda x$. However, the exact solution to~\eqref{eq:intro_main_equation} is: 
\begin{equation}\label{eq:exact}
u_\lambda(x)=\left\{
\begin{array}{rl}
x & \hbox{if } x\leq \lambda,\\
\lambda & \hbox{if } x\geq \lambda.
\end{array}\right.
\end{equation}
Hence, the relative error in the QoI for this case becomes:
\begin{equation}\label{eq:Galerkin_error}
\frac{|u_\lambda(\xnot)-u_h(\xnot)|}{|u_\lambda(\xnot)|}=\left\{
\begin{array}{cl}
1-\lambda & \hbox{if } \xnot\leq \lambda,\\
1-\xnot & \hbox{if } \xnot\geq \lambda,
\end{array}\right.
\end{equation}
As may be expected for this very coarse approximation, the relative errors are large (and actually never vanishes except in limiting cases).
\par
Let us instead consider a Petrov--Galerkin method for~\eqref{eq:intro_model}, with the same trial space~$\mbbU_h$, but a special test space~$\mbbV_h$, i.e., $u_h \in \mbbU_h:=\Span\{\psi\}$ such that $\int_{0}^{1} u_h'v_h' = v_h(\lambda)$, for all $v_h\in \mbbV_h:=\Span\{\varphi\}$. We use the parametrized test function 
$\varphi(x)=\theta_1x+ e^{-\theta_2}(1-e^{-\theta_1x})$, which is motivated by the simplest artificial neural network; see Section~\ref{sec:1D_diff} for  details.
By \emph{varying} the parameters $\theta_1,\theta_2\in\mbbR$, the errors in the quantity of interest can be significantly reduced. Indeed, Figure~\ref{fig:intro_example} shows the relative error in the QoI, plotted as a function of the $\theta_1$-parameter, with the other parameter set to $\theta_2=-9$, in the case of~$\xnot = 0.1$ and two values of~$\lambda$. 
When $\lambda=0.15 > 0.1=\xnot$ (left plot in Figure~\ref{fig:intro_example}), the optimal value $\theta_1\approx 48.5$ delivers a relative error of 0.575\% in the quantity of interest. Notice that the Galerkin method has a relative error $>80\%$. For $\lambda=0.05 < 0.1=\xnot$ (right plot in Figure~\ref{fig:intro_example}),
the value $\theta_1\approx 13.9$ actually delivers an \emph{exact} approximation of the QoI, while the Galerkin method has a relative error $\approx 90\%$. 
\par
This example illustrates a general trend that we have observed in our numerical test (see Section~\ref{sec:numTest}): Striking improvements in quantities of interest are achieved using well-tuned test spaces.
\begin{figure}[htp]
\begin{center}
 \begin{subfigure}[b]{0.32\textwidth}
    \includegraphics[width=\textwidth, height=\textwidth]{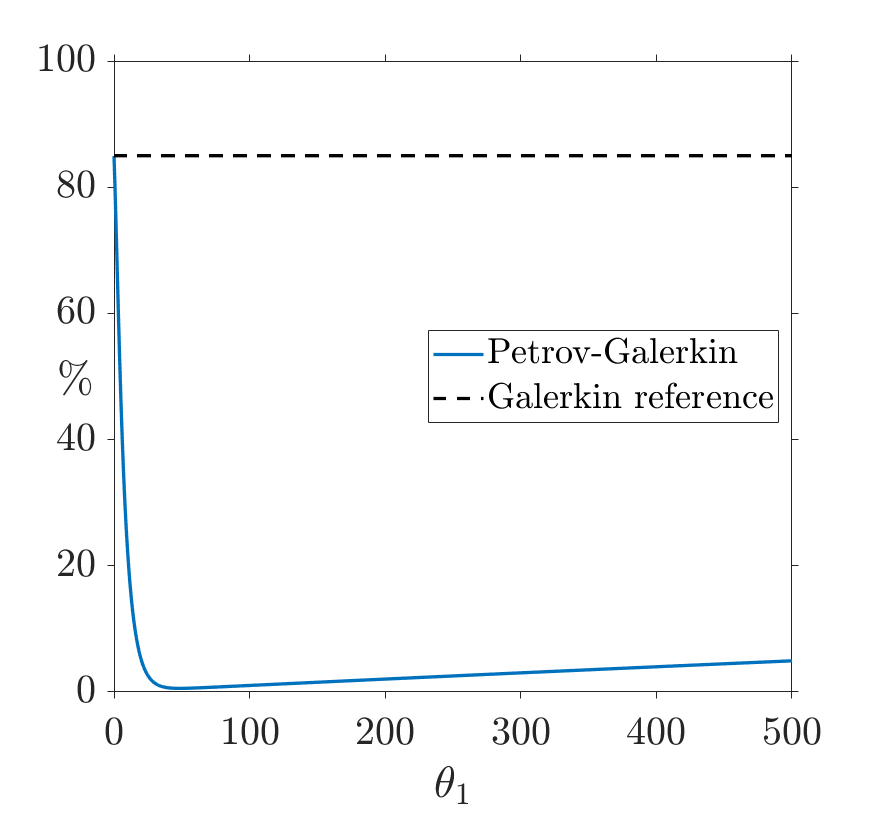}
    \caption{Relative error for $\lambda = 0.15$}
  \end{subfigure}
  \hspace{1cm}
  \begin{subfigure}[b]{0.32\textwidth}
	\includegraphics[width=\textwidth, height=\textwidth]{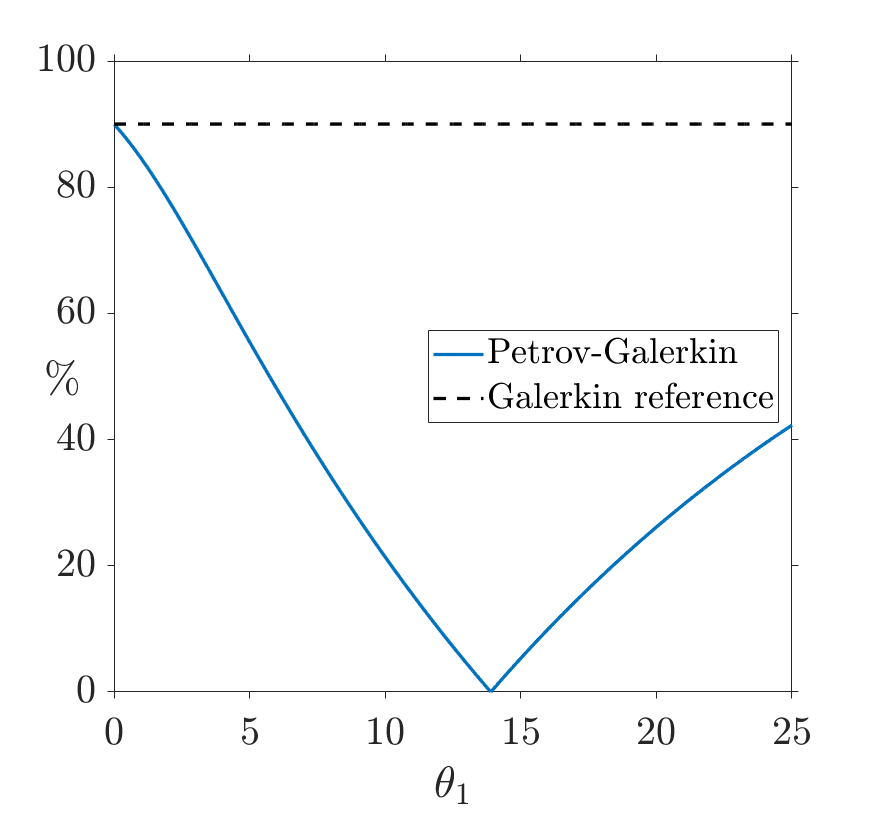}
    \caption{Relative error for $\lambda = 0.05$.}
  \end{subfigure}
  \caption{Relative error in the quantity of interest $\xnot = 0.1$, for different values of $\theta_1$.}
  \label{fig:intro_example}
\end{center}  
\end{figure}


\subsection{Related literature}
Let us note that deep learning, in the form of artificial neural networks, has become extremely popular in scientific computation in the past few years, a crucial feature being the capacity of neural networks to approximate any continuous function~\cite{cybenko1989approx}. While classical applications concern classification and prediction for image and speech recognition~\cite{goodfellow2016deep, lecun2015deep, higham2018deep}, there have been several new advances related to differential equations, either focussing on the data-driven discovery of governing equations~\cite{rudy2017data, berg2019data, RaiKarJCP2018} or the numerical approximation of (parametric) differential equations. 
\par
On the one hand, artificial neural networks can be directly employed to approximate a single PDE solution, see e.g.~\cite{BerNysNC2018,lagaris1998artificial,lee1990neural}, and in particular the recent high-dimensional Ritz method~\cite{EYuCMS2018}. On the other hand, in the area of model order reduction of differential equations, there have been tremendous recent developments in utilizing machine learning to obtain the reduced-order model for parametric models~\cite{ionita2014data,HesUbbJCP2018,regazzoni2019machine,SwiMaiPehWilCF2019,kutyniok2019theoretical}.
These developments are very closely related to recent works that use neural networks to optimize numerical methods, e.g., tuning the turbulence model~\cite{LinKurTemJFM2016}, slope limiter~\cite{RayHesJCP2018} or artificial viscosity~\cite{DisHesRayJCP2020}.
\par
The idea of goal-oriented adaptive (finite element) methods date back to the late 1990s, see e.g.,~\cite{BecRanAN2001,OdePruCMA2001,MomSteSINUM2009} for early works and analysis, and~\cite{FeiPraZeeSINUM2016,KerPruChaLafCMAME2017,BruZhuZwiCMAME2017,EndWicCMAM2017,HayKelRaiSunWesCAMCS2018} for some recent new developments. These methods are based on a different idea than the machine-learning framework that we propose. Indeed, the classical goal-oriented methods aim to adaptively refine the underlying meshes (or spaces) so as to control the error in the quantity of interest, thereby adding more degrees of freedom at each adaptive step. In our framework, we train a finite element method so as to control the error in the quantity of interest based on training data for a parametric model. In particular, we do not change the number of degrees of freedom. 
\subsection{Outline}
The contents of this paper are arranged as follows. Section~\ref{sec:method} presents the machine-learning methodology to constructing data-driven finite element methods. It also presents the stability analysis of the discrete method as well as equivalent discrete formulations. Section~\ref{sec:implement} presents several implementational details related to artificial neural networks and the training procedure. Section~\ref{sec:numTest} present numerical experiments for 1-D and 2-D elliptic and hyperbolic PDEs. Finally, Section~\ref{sec:concl} contains our conclusions.
%
\section{Methodology}
\label{sec:method}

\subsection{Abstract problem}

\noindent Let $\mbbU$ and $\mbbV$ be infinite dimensional Hilbert spaces spaces, with respective dual spaces $\mbbU^*$ and $\mbbV^*$. Consider a boundedly invertible linear operator $B:\mbbU\to \mbbV^{*}$, a family of right-hand-side functionals $\{\ell_\lambda\}_{\lambda\in \Lambda}\subset \mbbV^*$ that may depend non-affinely on~$\lambda$, and a quantity of interest functional $q\in \mbbU^*$. 
Given $\lambda\in \Lambda$, the continuous (or infinite-dimensional) problem will be to find $u_\lambda\in \mbbU$ such that:
\begin{equation}
\label{eq:abstract_eq}
Bu_\lambda = \ell_\lambda, \quad \text{in}\, \mbbV^{*},
\end{equation}
where the interest is put in the quantity $q(u_{\lambda})$. 
In particular, we consider the case when $\left<Bu,v\right>_{\mbbV^*,\mbbV}:=b(u,v)$, for a given bilinear form $b:\mbbU\times\mbbV\to\mbbR$. If so, problem~\eqref{eq:abstract_eq} translates into:
%
\begin{equation}
\label{eq:VarFor_gen_prob}
\left\{
\begin{array}{l}
\text{Find } u_\lambda\in \mbbU \text{ such that:} \\
b(u_\lambda,v)= \ell_{\lambda}(v),\qquad\forall v\in \mbbV, 
\end{array}
\right.
\end{equation}
which is a  type of problem that naturally arises in the context of variational formulations of partial differential equations with multiple right-hand-sides or \emph{parametrized PDEs}.%
\footnote{While parametrized bilinear forms $b_\lambda(\cdot,
\cdot)$ are also possible, they lead to quite distinct algorithmic details. We therefore focus on parametrized right-hand sides and leave parametrized bilinear forms for future work.}

\subsection{Main idea of the accelerated methods}

We assume that the space $\mbbV$ can be endowed with a family of equivalent {\it weighted} inner products $\{(\cdot,\cdot)_{\mbbV,\omega}\}_{\omega\in\mathcal W}$ and inherited norms $\{\|\cdot\|_{\mbbV,\omega}\}_{\omega\in\mathcal W}$, without affecting the topology given by the original norm $\|\cdot\|_\mbbV$ on $\mbbV$. That is, for each $\omega\in\mathcal W$, there exist equivalence constants $C_{1,\omega}>0$ and $C_{2,\omega}>0$ such that:
\begin{equation}\label{eq:norm_equiv}
C_{1,\omega}\|v\|_{\mbbV,\omega}\leq \|v\|_\mbbV\leq C_{2,\omega}\|v\|_{\mbbV,\omega}\,,\qquad
\forall v\in \mbbV.
\end{equation}
Consider a \emph{coarse} finite dimensional subspace $\mbbU_h\subset \mbbU$ where we want to approximate the solution of~\eqref{eq:VarFor_gen_prob}, and let $\mbbV_h\subset \mbbV$ be a discrete test space such that $\dim{\mbbV_h}\geq \dim\mbbU_h$. The discrete method that we want to use to approach the solution $u_\lambda\in\mbbU$ of problem~\eqref{eq:VarFor_gen_prob}, is to find $(r_{h,\lambda,\omega},u_{h,\lambda,\omega})\in\mbbV_h\times \mbbU_h$ such that:
%
\begin{subequations}
\label{eq:mixed_system}
\begin{empheq}[left=\left\{\;,right=\right.,box=]{alignat=3} 
\label{eq:mixed_system_a}
& (r_{h,\lambda,\omega},v_h)_{\mbbV,\omega}
+b(u_{h,\lambda,\omega},v_h)
&&=\ell_\lambda(v_h)
 & \quad  & \forall v_h\in \mbbV_h\,,
\\ \label{eq:mixed_system_b}
&  b(w_h,r_{h,\lambda,\omega})
&& =0  && \forall w_h\in \mbbU_h\,.
\end{empheq}
\end{subequations}
System~\eqref{eq:mixed_system} corresponds to a residual minimization in a discrete dual norm that is equivalent to a Petrov--Galerkin method. See Section~\ref{sec:analysis} for equivalent formulations and analysis of this discrete approach. In particular, the counterpart $r_{h,\lambda,\omega}\in\mbbV_h$ of the solution of~\eqref{eq:mixed_system} is interpreted as a minimal residual representative, while $u_{h,\lambda,\omega}\in\mbbU_h$ is the coarse approximation of $u_\lambda\in\mbbU$ that we are looking for. 

Assume now that one has a reliable sample set of $N_s\in\mbbN$ (precomputed) data $\{(\lambda_i,q(u_{\lambda_i}))\}_{i=1}^{N_s}$, where  $q(u_{\lambda_i})$ is either the quantity of interest of the exact solution of~\eqref{eq:VarFor_gen_prob} with $\lambda=\lambda_i\in\Lambda$, or else, a high-precision approximation of it.
The main goal of this paper is to find a particular weight $ \omega^*\in\mathcal W$, such that for the finite sample of parameters $\{\lambda_i\}_{i=1}^{N_s}\subset\Lambda$, the discrete solutions $\{u_{h,\lambda_i,\omega^*}\}_{i=1}^{N_s}\subset\mbbU_h$ of problem~\eqref{eq:mixed_system} with $\omega=\omega^*$, makes the errors in the quantity of interest as smallest as possible, i.e., 
\begin{equation}\label{eq:min}
\frac{1}{2}\sum_{i=1}^{N_s}\left| q(u_{\lambda_i}) - q(u_{h,\lambda_i,\omega^*})\right|^2\to\min.
\end{equation}
%
%
To achieve this goal we will work with a particular family of weights described by artificial neural networks (ANN). The particular {\it optimal} weight $\omega^*$ will be {\it trained} using data-driven algorithms that we describe in the following. Our methodology will be divided into an expensive 
{\it offline procedure} (see Section~\ref{sec:offline}) and an unexpensive {\it online procedure} (see Section~\ref{sec:online}).
%

In the offline procedure:
\begin{itemize}
\item A weight function $\omega^*\in\mathcal W$ that minimizes~\eqref{eq:min} for a sample set of training data $\{(\lambda_i,q(u_{\lambda_i}))\}_{i=1}^{N_s}$ is obtained.
\item From the matrix related with the discrete mixed formulation~\eqref{eq:mixed_system} using $\omega=\omega^*$, a {\it static condensation} procedure is applied to condense-out the residual variable $r_{h,\lambda,\omega^*}$. The condensed matrices are store for the online procedure.
\end{itemize}

In the online procedure:
\begin{itemize}
\item The condensed mixed system~\eqref{eq:mixed_system} with $\omega=\omega^*$ is solved for multiple right-hand-sides in $\{\ell_\lambda\}_{\lambda\in\Lambda}$, and the quantities of interest 
$\{q(u_{h,\lambda,\omega^*})\}_{\lambda\in\Lambda}$ are directly computed as reliable approximations of $\{q(u_\lambda)\}_{\lambda\in\Lambda}$.
\end{itemize}


\subsection{Analysis of the discrete method}\label{sec:analysis}
In this section we analyze the well-posedness of the discrete system~\eqref{eq:mixed_system}, as well as equivalent interpretations of it. The starting point will be always to assume well-posedness of the continuous (or infinite-dimensional) problem~\eqref{eq:VarFor_gen_prob}, which we will establish below.
\begin{theorem}\label{thm:infinite_wellposedness}
Let $(\mbbU,\|\cdot\|_\mbbU)$ and $(\mbbV,\|\cdot\|_\mbbV)$ be Hilbert spaces, and let $\|\cdot\|_{\mbbV,\omega}$ be the norm inherited from the weighted inner-product ${(\cdot,\cdot)_{\mbbV,\omega}}$, which satisfies the equivalence~\eqref{eq:norm_equiv}. Consider the problem~\eqref{eq:VarFor_gen_prob} and assume the existence of constants $M_\omega> 0$ and $\gamma_\omega> 0$
such that:
%
%
\begin{equation}\label{eq:B_constants}
\gamma_\omega \|u\|_\mbbU \leq \sup_{v\in \mbbV} \frac{|b(u,v)|}{\|v\|_{\mbbV,\omega}}\leq M_\omega \|u\|_\mbbU\,,\qquad\forall u\in\mbbU.
\end{equation}
Furthermore, assume that for any $\lambda\in\Lambda$:
\begin{equation}\label{eq:existence}
\left<\ell_\lambda,v\right>_{\mbbV^*,\mbbV}=0,\quad\forall v\in\mbbV \mbox{ such that } b(\cdot,v)=0\in\mbbU^*\,.
\end{equation}
Then, for any $\lambda\in\Lambda$, there exists a unique $u_\lambda\in \mbbU$ solution of problem~\eqref{eq:VarFor_gen_prob}.
\end{theorem}
\begin{proof}
This result is classical. Using operator notation (see eq.~\eqref{eq:abstract_eq}), condition~\eqref{eq:B_constants} says that the operator $B:\mbbU\to \mbbV^*$ such that $\left<Bu,v\right>_{\mbbV^*,\mbbV}=b(u,v)$ is continuous, injective and has a closed range. 
In particular, if $u_\lambda\in\mbbU$ exists, then it must be unique. 
The existence of $u_\lambda$ is guaranteed by condition~\eqref{eq:existence}, since $\ell_\lambda$ is orthogonal to the kernel of $B^*$, which means that $\ell_\lambda$ is in the range of $B$ by the Banach closed range theorem. 
\end{proof}
\begin{remark}
Owing to the equivalence of norms~\eqref{eq:norm_equiv}, if~\eqref{eq:B_constants} holds true for a particular weight $\omega\in\mathcal W$, then it also holds true for the original norm $\|\cdot\|_\mbbV$ of $\mbbV$, and for any other weighted norm linked to the family of weights $\mathcal W$.
\end{remark}

The next Theorem~\ref{thm:discrete_wellposedness} establishes the well-posedness of the discrete mixed scheme~\eqref{eq:mixed_system}.

\begin{theorem}\label{thm:discrete_wellposedness}
Under the same assumptions of Theorem~\ref{thm:infinite_wellposedness}, let $\mbbU_h\subset \mbbU$ and $\mbbV_h\subset \mbbV$ be finite dimensional subspaces such that $\dim{\mbbV_h}\geq \dim\mbbU_h$, and such that the following discrete inf-sup conditon is satisfied:
\begin{equation}\label{eq:inf-sup}
\sup_{v_h\in\mbbV_h}{|b(u_h,v_h)|\over \|v_h\|_{\mbbV,\omega}}\geq \gamma_{h,\omega}\|u_h\|_\mbbU\,, \qquad\forall u_h\in\mbbU_h\,,
\end{equation}
where $\gamma_{h,\omega}>0$ is the associated discrete {\it inf-sup constant}. Then, the mixed system~\eqref{eq:mixed_system} has a unique solution $(r_{h,\lambda,\omega},u_{h,\lambda,\omega})\in\mbbV_h\times\mbbU_h$. Moreover, $u_{h,\lambda,\omega}$ satifies the a~priori estimates:
\begin{equation}\label{eq:apriori_estimates}
\|u_{h,\lambda,\omega}\|_\mbbU \leq {M_\omega\over \gamma_{h,\omega}}\|u_\lambda\|_\mbbU
\quad\mbox{ and }\quad
\|u_\lambda - u_{h,\lambda,\omega}\|_\mbbU \leq {M_\omega\over \gamma_{h,\omega}}\inf_{u_h\in\mbbU_h}\|u_\lambda-u_h\|_\mbbU.
\end{equation}
\end{theorem}
\begin{proof}
See Appendix~\ref{sec:proof}.
\end{proof}
\begin{remark}
It is straightforward to see, using the equivalences of norms~\eqref{eq:norm_equiv}, that having the discrete inf-sup condition in one weighted norm $\|\cdot\|_{\mbbV,\omega}$ is fully equivalent to have the discrete inf-sup condition in the original norm of $\mbbV$, and also to have the discrete inf-sup condition in another weighted norm linked to the family of weights $\mathcal W$. If~\eqref{eq:inf-sup} holds true for any weight of the family $\mathcal W$ (or for the original norm of $\mbbV$) we say that the the discrete pairing $\mbbU_h$-$\mbbV_h$ is compatible.
\end{remark}
\begin{remark}[Influence of the weight]
In general, to make the weight $\omega\in\mathcal W$ influence the  mixed system~\eqref{eq:mixed_system}, we need $\dim{\mbbV_h} > \dim\mbbU_h$. In fact, the case $\dim{\mbbV_h} = \dim\mbbU_h$ is not interesting because equation~\eqref{eq:mixed_system_b} becomes a square system and one would obtain $r_{h,\lambda,\omega}=0$ from it, thus recovering a standard Petrov-Galerkin method without any influence of~$\omega$.
\end{remark}
\subsubsection{Equivalent Petrov-Galerkin formulation}\label{sec:PGequiv}
For any weight $\omega\in\mathcal W$, consider the trial-to-test operator $T_\omega:\mbbU\to \mbbV$ such that:
\begin{equation}\label{eq:T_omega}
(T_\omega u,v)_{\mbbV,\omega}= b(u,v)\,,\qquad \forall u \in \mbbU,\, \forall v\in\mbbV.
\end{equation}
Observe that for any $u\in \mbbU$, the vector $T_\omega u\in\mbbV$ is nothing but the Riesz representative of the functional $b(u,\cdot)\in \mbbV^*$ under the weighted inner-product $(\cdot,\cdot)_{\mbbV,\omega}\,$. 

Given a discrete subspace $\mbbU_h\subset \mbbU$, the optimal test space paired with $\mbbU_h$, is defined as $T_\omega\mbbU_h\subset \mbbV$. 
The concept of optimal test space was introduced by~\cite{demkowicz2010class} and its main advantage is that the discrete pairing 
$\mbbU_h$-$T_\omega\mbbU_h$ (with equal dimensions) satisfies automatically the inf-sup condition~\eqref{eq:inf-sup}, with inf-sup constant $\gamma_\omega >0$, inherited from the stability at the continuous level (see eq.~\eqref{eq:B_constants}). 

Of course, equation~\eqref{eq:T_omega} is infinite dimensional and thus not solvable in the general case. Instead, having the discrete finite-dimensional subspace $\mbbV_h\subset \mbbV$, we can define the discrete trial-to-test operator 
$T_{h,\omega}:\mbbU\to \mbbV_h$ such that:
\begin{alignat}{2}\label{eq:T_h,omega}
(T_{h,\omega} u,v_h)_{\mbbV,\omega}= b(u,v_h)\,, \quad\forall u\in\mbbU,\,\forall v_h\in\mbbV_h.
\end{alignat}
Observe now that the vector $T_{h,\omega} u\in\mbbV_h$ corresponds to the orthogonal projection of $T_\omega u$ into the space $\mbbV_h$, by means of the weighted inner-product $(\cdot,\cdot)_{\mbbV,\omega}$.  This motivates the definition of the {\it projected optimal test space} (of the same dimension of $\mbbU_h$) as $\mbbV_{h,\omega}:=T_{h,\omega}\mbbU_h$ (cf.~\cite{BroSteCAMWA2014}). It can be proven that if the discrete pairing $\mbbU_h$-$\mbbV_h$ satisfies the inf-sup condition~\eqref{eq:inf-sup}, then the discrete pairing $\mbbU_h$-$\mbbV_{h,\omega}$ also satisfies the inf-sup conditon~\eqref{eq:inf-sup}, with the same inf-sup constant $\gamma_{h,\omega}>0$.
Moreover, the solution $u_{h,\lambda,\omega}\in \mbbU_h$ of the mixed system~\eqref{eq:mixed_system} is also the unique solution of the well-posed Petrov-Galerkin scheme:
\begin{equation}\label{eq:PG}
b(u_{h,\lambda,\omega},v_h)=\ell_\lambda(v_h),\qquad\forall v_h\in\mbbV_{h,\omega}\,.
\end{equation}
Indeed, from equation~\eqref{eq:mixed_system_b}, for any $v_h=T_{h,\omega}w_h\in\mbbV_{h,\omega}\subset\mbbV_h$, we obtain that
$$(r_{h,\lambda,\omega},v_h)_{\mbbV,\omega}=(r_{h,\lambda,\omega},T_{h,\omega}w_h)_{\mbbV,\omega}=b(w_h,r_{h,\lambda,\omega})=0,$$
which upon being replaced in equation~\eqref{eq:mixed_system_a} of the mixed system gives~\eqref{eq:PG}. We refer to~\cite[Proposition~2.2]{BroSteCAMWA2014} for further details.  


\subsubsection{Equivalent Minimal Residual formulation}\label{sec:MinRes}
Let $\mbbU_h\subset \mbbU$ and $\mbbV_h\subset\mbbV$ as in Theorem~\ref{thm:discrete_wellposedness}, and consider the following discrete-dual residual minimization:
\begin{equation}\notag
 u_{h,\lambda,\omega}=\argmin_{w_h\in \mbbU_h}\|\ell_\lambda(\cdot)-b(w_h,\cdot)\|_{(\mbbV_h)^*_\omega}\,\,\,,\quad\mbox{where } \|\cdot\|_{(\mbbV_h)^*_\omega}:=\sup_{v_h\in\mbbV_h} {|\left<\,\cdot\,,v_h\right>_{\mbbV^*,\mbbV}|\over \|v_h\|_{\mbbV,\omega}}.
\end{equation}
Let $R_{\omega,\mbbV_h}:\mbbV_h\to (\mbbV_h)^*$ be the Riesz map (isometry) linked to the weighted inner-product $(\cdot,\cdot)_{\mbbV,\omega}$, that is: 
$$
\left<R_{\omega,\mbbV_h}v_h\,,\,\cdot\,\right>_{(\mbbV_h)^*,\mbbV_h}=(v_h\,,\,\cdot\, )_{\mbbV,\omega} \,,\quad\forall v_h\in\mbbV_h.
$$
Defining the minimal residual representative
$r_{h,\lambda,\omega}:=R^{-1}_{\omega,\mbbV_h}(\ell_\lambda(\cdot)-b(u_{h,\lambda,\omega},\cdot))$, we observe that the couple $(r_{h,\lambda,\omega},u_{h,\lambda,\omega})\in\mbbV_h\times\mbbU_h$ is the solution of the mixed system~\eqref{eq:mixed_system}. Indeed,  $r_{h,\lambda,\omega}\in\mbbV_h$ satisfy:
\begin{equation}\notag
(r_{h,\lambda,\omega},v_h)_{\mbbV,\omega}= \ell_\lambda(v_h)-b(u_{h,\lambda,\omega},v_h)\,,\quad\forall v_h\in\mbbV_h\,,
\end{equation}
which is nothing but equation~\eqref{eq:mixed_system_a} of the mixed system.
On the other hand, using the isometric property of $R_{\omega,\mbbV_h}$ we have:
$$
u_{h,\lambda,\omega}=\argmin_{w_h\in \mbbU_h}\left\|\ell_\lambda(\cdot)-b(w_h,\cdot)\right\|^2_{(\mbbV_h)^*_\omega}
=\argmin_{w_h\in \mbbU_h}\left\|R^{-1}_{\omega,\mbbV_h}(\ell_\lambda(\cdot)-b(w_h,\cdot))\right\|^2_{\mbbV,\omega}\,.
$$
Differentiating the norm $\|\cdot\|_{\mbbV,\omega}$ and using first-order optimality conditions we obtain:
$$
0 = \left(R^{-1}_{\omega,\mbbV_h}(\ell_\lambda(\cdot)-b(u_{h,\lambda,\omega},\cdot)),R^{-1}_{\omega,\mbbV_h}b(w_h,\cdot)\right)_{\mbbV,\omega}=b(w_h,r_{h,\lambda,\omega}),\quad
\forall w_h\in \mbbU_h\,,
$$
which gives equation~\eqref{eq:mixed_system_b}. 

\section{Implementational details}
\label{sec:implement}

\subsection{Artificial Neural Networks}\label{sec:ANN}

\noindent Roughly speaking, an artificial neural network is obtained from compositions and superpositions of a single, simple nonlinear activation or response function \cite{cybenko1989approx}. Namely, given an input $x_{\mathrm{in}}\in\mbbR^d$ and an activation function $\sigma$, an artificial neural network looks like:
\begin{equation}
\label{eq:general_ANN}
\ANN(x_{\mathrm{in}}) = \Theta_n\sigma(\cdots\sigma( \Theta_2\sigma(\Theta_1 x_{\mathrm{in}} + \phi_1) + \phi_2)\cdots) + \phi_n,
\end{equation}
where $\{\Theta_j\}_{j=1}^n$ are matrices (of different size) and $\{\phi_j\}_{j=1}^n$ are vectors (of different length) of coefficients to be determined by a ``training'' procedure. Depending on the application, an extra activation function can be added at the end. A classical activation function is the logistic sigmoid function:
\begin{equation}\label{eq:sigmoid}
\sigma(x) = \frac{1}{1+ e^{-x}}\,.
\end{equation}
Other common activation functions used in artificial neural network applications are the rectified linear unit (ReLU), the leaky ReLU, and the hyperbolic tangent (see, e.g.\cite{chien2018source, tsihrintzis2019machine}).

 The process of training an artificial neural network as  (\ref{eq:general_ANN}) is performed by the minimization of a given functional $J(\Theta_1,\phi_1,\Theta_2,\phi_2,\dots,\Theta_n,\phi_n)$. We search for optimal sets of parameters $\{\Theta_j^{*}\}_{j=1}^n$ and $\{\phi_j^{*}\}_{j=1}^n$ minimizing the cost functional $J$. For simplicity, in what follows we will denote all the parameters of an artificial neural network by $\theta\in\Phi$, for a given set $\Phi$ of admissible parameters.
A standard cost functional is constructed with a sample training set of known values $\{x_1,x_2,\dots, x_{N_s}\}$ and its corresponding labels $\{y_1,y_2,\dots,y_{N_s}\}$ as follows:
\begin{equation}
\nonumber
J(\theta) = \frac{1}{2} \sum_{i=1}^{N_s} \left( y_i - F(\ANN(x_i;\theta) )\right)^2,
\end{equation}   
(for some real function $F$) 
which is known as supervised learning \cite{goodfellow2016deep}.
Training an artificial neural network means to solve the following minimization problem:
\begin{equation}
\label{eq:argmin}
\theta^*=\argmin_{\theta\in\Phi} J(\theta).
\end{equation}
Thus, the artificial neural network evaluated in the optimal $\theta^{*}$ (i.e., $\ANN(x;\theta^{*})$) is the trained network. There are many sophisticated tailor-made procedures to perform the minimization 
in~\eqref{eq:argmin} efficiently. The reader may refer to~\cite{sra2011optimization} for inquiring into this topic, which is out of the scope of this paper.


\subsection{Offline procedures}\label{sec:offline}

The first step is to choose an artificial neural network $\ANN(\cdot\,; \theta)$ that will define a family $\mathcal W$ of positive weight-functions to be used in the weighted inner products $\{(\cdot,\cdot)_{\mbbV,\omega}\}_{\omega\in\mathcal W}$. Typically we have:
$$
\mathcal W=\{\omega(\cdot)=g(\ANN(\cdot\,; \theta)): \theta\in \Phi\}, 
$$
where $g$ is a suitable positive an bounded continuous function.

Next, given a discrete trial-test pairing $\mbbU_h$-$\mbbV_h$ satisfying~\eqref{eq:inf-sup}, we construct the map $\mathcal W\times\Lambda\ni(\omega,\lambda)\mapsto q(u_{h,\lambda,\omega})\in\mbbR$, 
where $u_{h,\lambda,\omega}\in\mbbU_h$ is the second component of the solution the mixed system~\eqref{eq:mixed_system}. 
%
%
Having coded this map, we proceed to {\it train} the $\ANN$ by computing:
\begin{equation}
\label{eq:eqmin_prob_mixed}
\left\{
\begin{array}{l}
\theta^*=\displaystyle \argmin_{\theta\in\Phi}  \displaystyle {1\over 2}\sum_{i=1}^n\left| q(u_{h,\lambda_i,\omega})-q(u_{\lambda_i})\right |^2,\\
\mbox{s.t.} 
\left\{
\begin{array}{lll}
 \omega(\cdot) = g(\ANN(\cdot\,;\theta))\\
 (r_{h,\lambda_i,\omega},v_h)_{\mbbV,\omega} + b(u_{h,\lambda_i,\omega},v_h) & = \ell_{\lambda_i}(v_h), & \forall\, v_h\in\mbbV_h,\\
b(w_h,r_{h,\lambda_i,\omega}) & =0,  & \forall\, w_h\in \mbbU_h.
\end{array}
\right.
\end{array}
\right.
\end{equation}

The last step is to build the matrices of the linear system needed for the online phase. Denote the basis of $\mbbU_h$ by $\{\psi_1,...,\psi_n\}$,  and the basis of $\mbbV_h$ by $\{\varphi_1,...,\varphi_m\}$ 
 (recall that $m>n$).
Having $\theta^*\in\Phi$ approaching~\eqref{eq:eqmin_prob_mixed}, we extract from the mixed system~\eqref{eq:mixed_system} the matrices $A\in \mbbR^{m\times m}$ and $B\in \mbbR^{m\times n}$ such that:
$$
A_{ij} = (\varphi_j,\varphi_i)_{\mbbV,\omega^*}\qquad\mbox{and}\qquad  
B_{ij} = b(\psi_j,\varphi_i),
$$
where $\omega^*(\cdot)=g(\ANN(\cdot\,;\theta^*))$.
Finally, we store the matrices $B^TA^{-1}B\in\mbbR^{n\times n}$ and $B^TA^{-1}\in\mbbR^{n\times m}$ to be used in the online phase to compute directly $u_{h,\lambda,\omega^*}\in \mbbU_h$ for any right hand side $\ell_\lambda\in\mbbV^*$.
Basically, we have condensed-out the residual variable of the mixed system~\eqref{eq:mixed_system}, since it is useless for the application of the quantity of interest $q\in\mbbU^*$. In addition, it will be also important to store the vector $Q\in\mbbR^n$ such that: 
$$
Q_j := q(\psi_j)\,,\quad j=1,...,n.
$$
\subsection{Online procedures}\label{sec:online}
For each $\lambda\in\Lambda$ for which we want to obtain the quantity of interest $q(u_{h,\lambda,\omega^*})$, we first compute the vector $L_\lambda\in  \mbbR^{m}$ such that its $i$-th component is given by:
$$
(L_\lambda)_i = \left<\ell_\lambda,\varphi_i\right>_{\mbbV^*,\mbbV}\,,
$$
where $\varphi_i$ is the $i$-th vector of in the basis of $\mbbV_h$. Next, we compute
$$
q(u_{h,\lambda,\omega^*})=Q^T(B^TA^{-1}B)^{-1}B^TA^{-1}L_\lambda.
$$
Observe that the matrix $Q^T(B^TA^{-1}B)^{-1}B^TA^{-1}$ can be fully obtained and stored from the previous offline phase (see Section~\ref{sec:offline}).
\section{Numerical tests}
\label{sec:numTest}

In this section, we show some numerical examples in 1D and 2D to investigate the main features of the proposed data-driven finite element method. In particular, we consider in the following order: 1D diffusion, 1D advection, 1D advection with multiple QoIs, and finally 2D diffusion. 

\subsection{1D diffusion with one QoI}\label{sec:1D_diff}
We recover here the motivational example from the introduction (see Section~\ref{sec:motivation}).
Consider the variational formulation~\eqref{eq:intro_model}, with trial and test spaces $\mbbU=\mbbV=H_{(0}^{1}(0,1)$. We endowed $\mbbV$ with the weighted inner product: 
$$
(v_1,v_2)_{\mbbV,\omega}:=\int_0^1 \omega \, v_1' v_2'\,,\qquad\forall v_1,v_2\in\mbbV.
$$
As in the introduction, we consider the simplest coarse discrete trial space $\mbbU_h:=\Span\{\psi\}\subset\mbbU$, where $\psi(x)=x$.
The optimal test function (see Section~\ref{sec:PGequiv}), paired with the trial function $\psi$, is 
given by $\varphi:=T_\omega\psi\in\mbbV$, which is the solution of~\eqref{eq:T_omega} with $u=\psi$. Hence,
\begin{equation}\label{eq:optimal_test}
\varphi(x)=\int_0^x {1\over\omega(s)}ds.
\end{equation}
%

Let us consider the Petrov-Galerkin formulation with \emph{optimal} test functions, which is equivalent to the 
mixed system~\eqref{eq:mixed_system} in the optimal case $\mbbV_h=\mbbV$. 
Consequently, the Petrov-Galerkin scheme with trial function $\psi$ and optimal test function $\varphi$, delivers the discrete solution
$u_{h,\lambda,\omega}(x)=x\varphi(\lambda)/\varphi(1)$ (notice that the trivial weight $\omega\equiv 1$ recovers the test function $\varphi=\psi$, and therefore the standard Galerkin approach).

Recalling the exact solution~\eqref{eq:exact}, we observe that the relative error in the quantity of interest for our Petrov-Galerkin approach is:
\begin{equation}\label{eq:rel_error}
\mbox{Err} = \left\{
\begin{array}{ll}
\left|1- {\varphi(\lambda)\over\varphi(1)}\right| & \mbox{if } \xnot\leq\lambda,\\
\left|1 - {\xnot\over\lambda} {\varphi(\lambda)\over\varphi(1)}\right| & \mbox{if } \xnot\geq\lambda.
\end{array}\right.
\end{equation}
Of course, any function such that $\varphi(\lambda)=\varphi(\xnot)\neq 0$ for $\lambda\geq \xnot$, and 
$\varphi(\lambda)=\lambda \varphi(\xnot)/\xnot$ for $\lambda\leq \xnot$, will produce \emph{zero error for all $\lambda\in(0,1)$}.
Notice that such a function indeed exists, and in this one-dimensional setting it solves the adjoint problem:
\begin{equation}
\notag
\left\{
\begin{array}{l}
\text{Find } z\in H_{(0}^{1}(0,1) \text{ such that:} \\
\displaystyle \int_{0}^{1} w'z' = w(\xnot), \qquad\forall w\in H_{(0}^{1}(0,1). 
\end{array}
\right.
\end{equation}
This optimal test function is also obtained in our framework via~\eqref{eq:optimal_test}, by using a limiting weight of the form:
\begin{equation}\label{eq:weight}
\omega(x) \rightarrow \left\{
\begin{array}{cl}
c & \mbox{if } x<\xnot,\\
+\infty & \mbox{if } x>\xnot,
\end{array}\right.
\end{equation}
for some constant $c>0$. Hence, the Petrov--Galerkin method using a test function of the form~\eqref{eq:optimal_test} has sufficient variability to eliminate any errors for any~$\lambda$!

We now restrict the variability by parametrizing $\omega$. In the motivating example given in Section~\ref{sec:motivation}, for illustration reasons we chose a weight of the form $\omega(x)=\sigma(\theta_1x+\theta_2)$, which corresponds to the most simple artificial neural network, having only one hidden layer with one neuron. We now select a slightly more complex family of weights having the form $\omega(x)=\exp( \ANN(x;\theta))>0$, where
\begin{equation}\label{eq:ANN1D}
\mathrm{ANN}(x;\theta) = \sum_{j=1}^{5} \theta_{j3} \sigma(\theta_{j1}x + \theta_{j2}).
\end{equation}
Observe that $\ANN(x;\theta)$ corresponds to an artificial neural network of one hidden layer with five neurons (see Section~\ref{sec:ANN}). 
\par
The training set of parameters has been chosen as $\lambda_i=0.1i$, with $i=1,...,9$.
For comparisons, we perform three different experiments. The first experiment trains the network~\eqref{eq:ANN1D} based on a cost functional that uses the relative error formula~\eqref{eq:rel_error}, where the optimal test function $\varphi$ is computed using eq.~\eqref{eq:optimal_test}. The other two experiments use the training approach~\eqref{eq:eqmin_prob_mixed}, with discrete spaces $\mbbV_h$ consisting of conforming piecewise linear functions over uniform meshes of 4 and 16 elements respectively. The quantity of interest has been set to $\xnot=0.6$, which does not coincide with a node of the discrete test spaces. Figure~\ref{fig:1d_diffusion_solutions} shows the obtained discrete solutions 
$u_{h,\lambda,\omega^*}$ for each experiment, and for two different values of $\lambda$.
Figure~\ref{fig:1d_diffusion_weight_and_error_a} shows the trained weight obtained for each experiment (cf.~eq.~\eqref{eq:weight}), while Figure~\ref{fig:1d_diffusion_weight_and_error_b} depicted the associated optimal and projected-optimal test functions linked to those trained weights. Finally, 
Figure~\ref{fig:1d_diffusion_weight_and_error_c} shows the relative errors in the quantity of interest for each discrete solution in terms of the $\lambda$ parameter. 
\par
It can be observed that the trained method using a parametrized weight function based on~\eqref{eq:ANN1D}, while consisting of only one degree of freedom, gives quite accurate quantities of interest for the entire range of~$\lambda$. This should be compared to the $O(1)$ error for standard Galerkin given by~\eqref{eq:Galerkin_error}. We note that some variation can be observed depending on whether the optimal or a projected optimal test function is used (with a richer~$\mbbV_h$ being better).
\begin{figure}[!t]
\begin{center}
  \begin{subfigure}[b]{0.32\textwidth}
    \includegraphics[width=\textwidth, height=\textwidth]{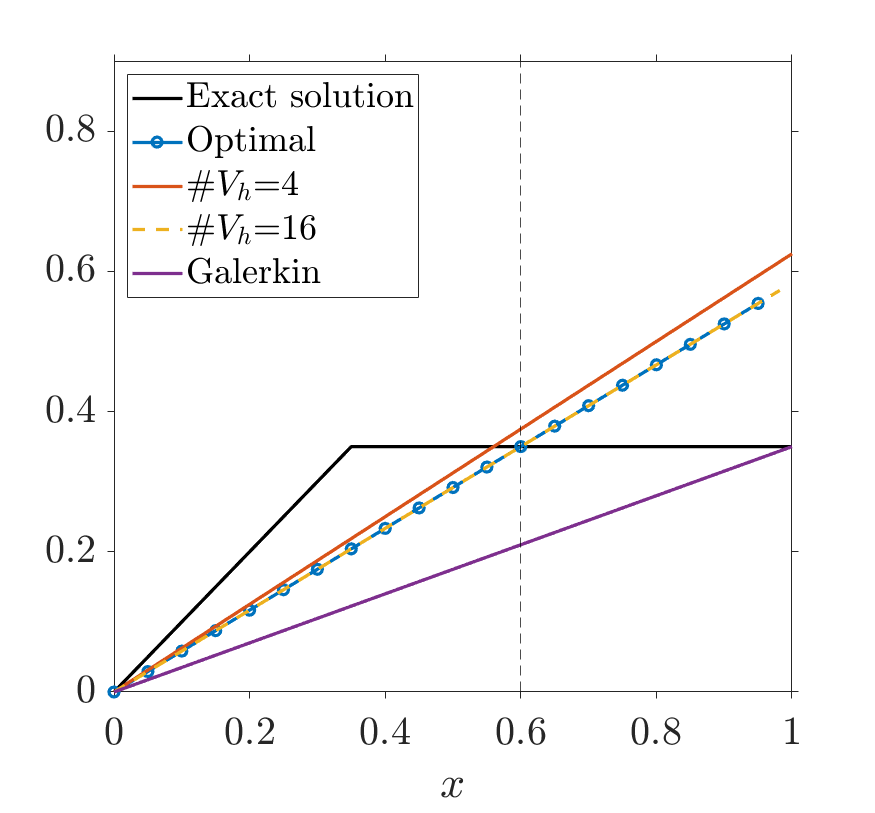}
    \caption{$\lambda = 0.35$}
  \end{subfigure}
  \hspace{1cm}
  \begin{subfigure}[b]{0.32\textwidth}
    \includegraphics[width=\textwidth, height=\textwidth]{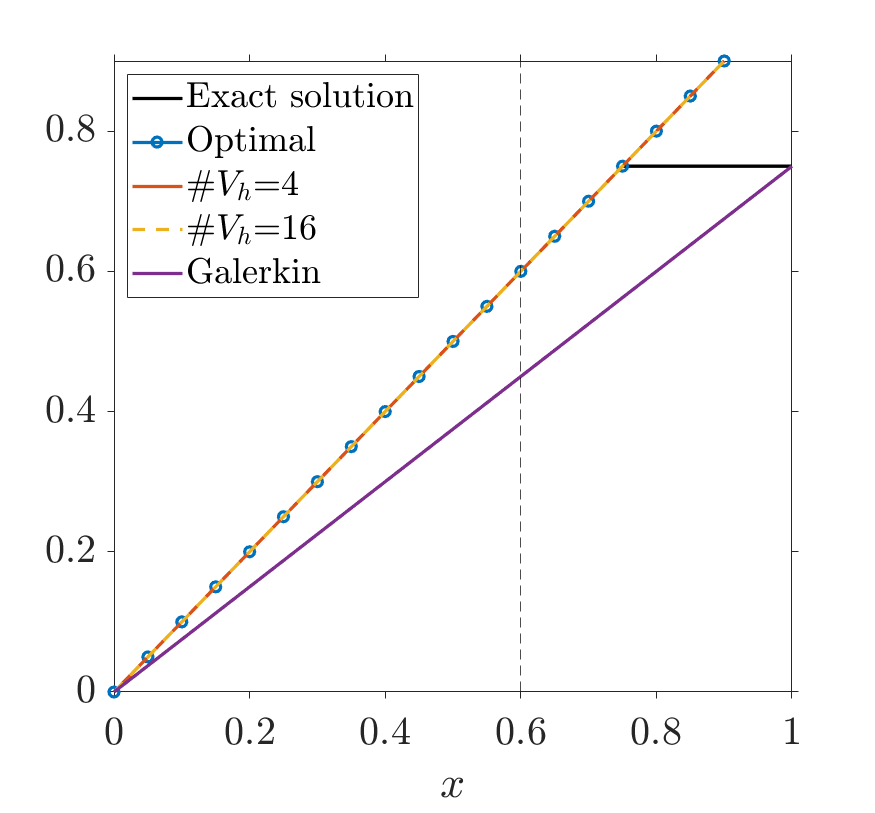}
    \caption{$\lambda = 0.75$}
  \end{subfigure}
\end{center}  
  \caption{Discrete solutions computed using the optimal test function approach (blue line), and discrete mixed form approach~\eqref{eq:mixed_system} with different discrete test spaces $\mbbV_h$ (red and yellow lines). Dotted line shows the QoI location.}
  \label{fig:1d_diffusion_solutions}
\end{figure}
\begin{figure}[!t]
\begin{center}
  \begin{subfigure}[b]{0.32\textwidth}
    \includegraphics[width=\textwidth, height=\textwidth]{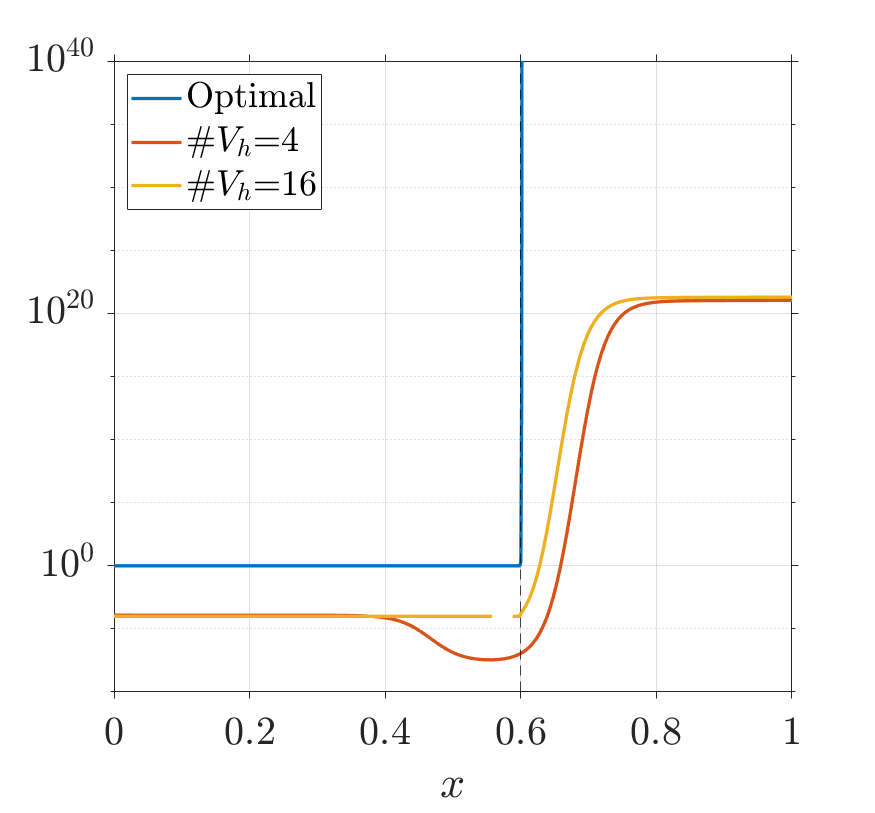}
    \caption{Trained weights}
    \label{fig:1d_diffusion_weight_and_error_a}
  \end{subfigure}
  \hfill
  \begin{subfigure}[b]{0.32\textwidth}
    \includegraphics[width=\textwidth, height=\textwidth]{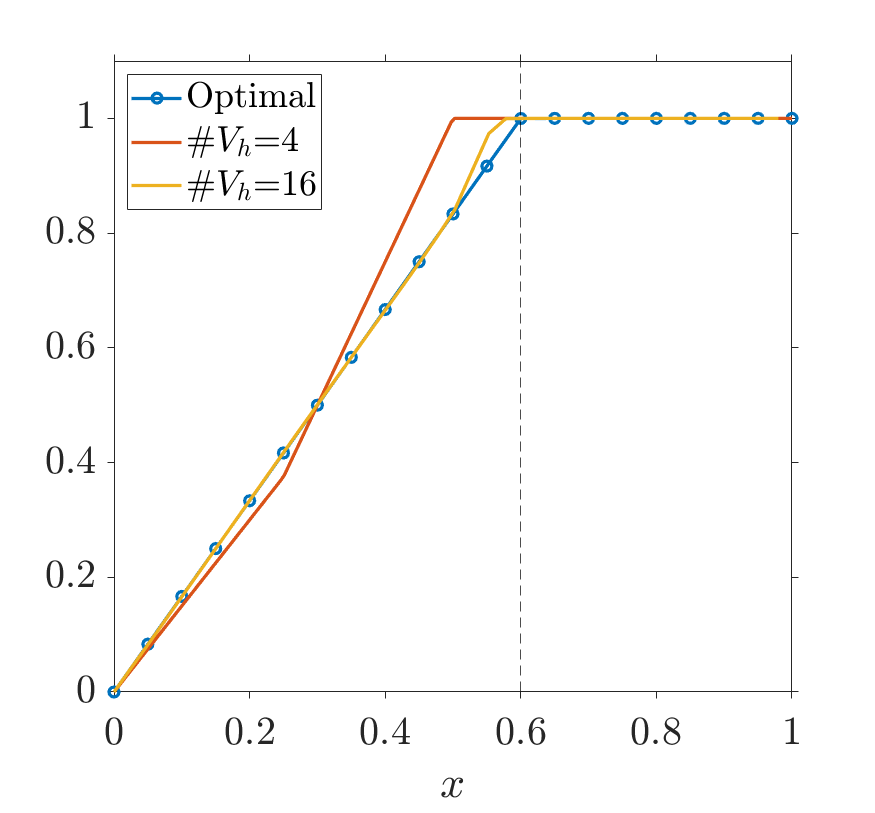}
    \caption{Optimal test functions}
    \label{fig:1d_diffusion_weight_and_error_b}
  \end{subfigure}
  \hfill
  \begin{subfigure}[b]{0.32\textwidth}
    \includegraphics[width=\textwidth, height=\textwidth]{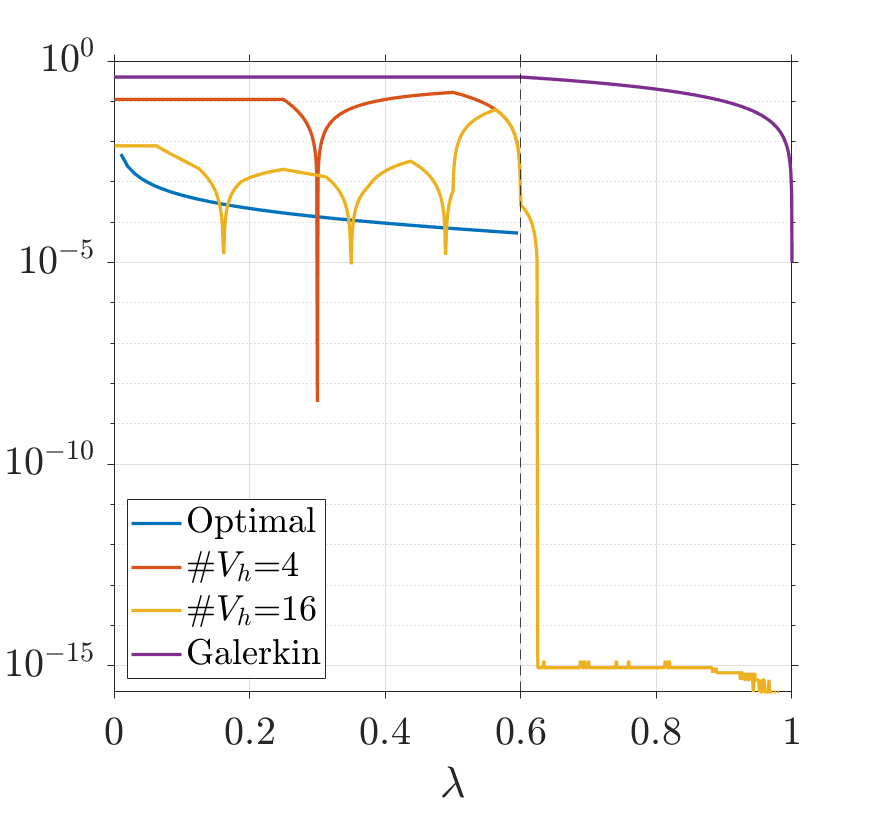}
    \caption{Relative errors in QoI}
    \label{fig:1d_diffusion_weight_and_error_c}
  \end{subfigure}
\end{center}  
  \caption{Trained weights, optimal (and projected-optimal) test functions, and relative errors computed with three different approaches. Dotted line shows the QoI location.}
  \label{fig:1d_diffusion_weight_and_error}
\end{figure}

\subsection{1D advection with one QoI}\label{sec:1D_adv}
Consider the family of ODEs:
\begin{equation}\label{eq:ode_example1}
\left\{
\begin{array}{r@{}lr}
u'&{}=f_{\lambda} & \hbox{in } (0,1),\\
u(0) &{}=0, &
\end{array}
\right.
\end{equation}
for a family of continuous functions $\{f_{\lambda}\}_{\lambda\in[0,1]}$ given by 
$f_\lambda(x):=(x-\lambda)\mathbbm{1}_{[\lambda,1]}(x)$, where $\mathbbm{1}_{[\lambda,1]}$ denote the characteristic function of the interval $[\lambda,1]$.
The exact solution of~\eqref{eq:ode_example1} will be used as a reference solution and is given by
$u_\lambda(x)={1\over 2}(x-\lambda)^2\mathbbm{1}_{[\lambda,1]}(x)$.
The quantity of interest considered for this example will be $q_{\xnot}(u_\lambda) := u_\lambda(\xnot)$, where $\xnot$ could be any value in $[0,1]$.

Let us consider the following variational formulation of problem~\eqref{eq:ode_example1}:
\begin{equation}
\nonumber
\left\{
\begin{array}{l}
\text{Find } u_\lambda\in \mbbU \text{ such that:} \\
b(u_\lambda,v):=\displaystyle\int_0^1u_\lambda'v = \int_0^1 f_{\lambda}v=:\ell_{\lambda}(v),\quad\forall v\in \mbbV, 
\end{array}
\right.
\end{equation}
where $\mbbU:= H^1_{(0}(0,1) := \{u\in L^2(0,1): u'\in L^2(0,1) \wedge u(0)=0\}$, and $\mbbV := L^2(0,1)$ 
is endowed with the weighted inner-product:
$$ (v_1,v_2)_{\mathbb V,\omega}:=\int_0^1 \omega\,v_1v_2\,, \quad\forall v_1,v_2\in\mbbV.$$
We want to approach this problem using coarse discrete trial spaces $\mathbb U_h\subset \mbbU$ of piecewise linear polynomials on a partition of one, two and three elements. 

We describe the weight $\omega(x)$ by the sigmoid of an artificial neural network that depends on parameters $\theta$, i.e., $\omega(x)= \sigma(\mathrm{ANN}(x;\theta))>0$ (see Section~\ref{sec:ANN}). In particular, we use the artificial neural network given in~\eqref{eq:ANN1D}.
%
%
%
%
To train such a network, we consider a training set $\{\lambda_i\}_{i=1}^{9}$, where $\lambda_i = 0.125(i-1)$,
together with the set of exact quantities of interest $\{q_{\xnot}(u_{\lambda_i})\}_{i=1}^{9}$, computed using the reference exact solution with $\xnot=0.9$.
%
%
The training procedure uses the constrained minimization problem~\eqref{eq:eqmin_prob_mixed}, where for each low-resolution trial space $\mbbU_h$ (based on one, two and three elements), the same discrete test space $\mbbV_h$ has been used: a high-resolution space of piecewise linear and continuous functions linked to a uniform partition of 128 elements.    
%
%
The minimization algorithm has been stopped until the cost functional reaches the tolerance \texttt{tol}$=9\cdot 10^{-7}$. 

After an optimal parameter $\theta^*$ has been found (see~\eqref{eq:eqmin_prob_mixed}), we follow the matrix procedures described in Section~\ref{sec:offline} and Section~\ref{sec:online} to approach the quantity of interest of the discrete solution for any $\lambda\in[0,1]$.

Figures~\ref{fig:1d_solutions}~and~\ref{fig:1d_error} show numerical experiments considering model problem (\ref{eq:ode_example1}) in three different trial spaces.  
Figure~\ref{fig:1d_solutions} shows, for $\lambda=0.19$, the exact solution and the Petrov-Galerkin solution computed with projected optimal test functions given by the trained weighted inner-product. Notice that for the three cases (with one, two, and three elements) the Petrov-Galerkin solution intends to approximate the quantity of interest 
(dotted line). 
\begin{figure}[!t]
\begin{subfigure}[b]{0.32\textwidth}
    \includegraphics[width=\textwidth, height=\textwidth]{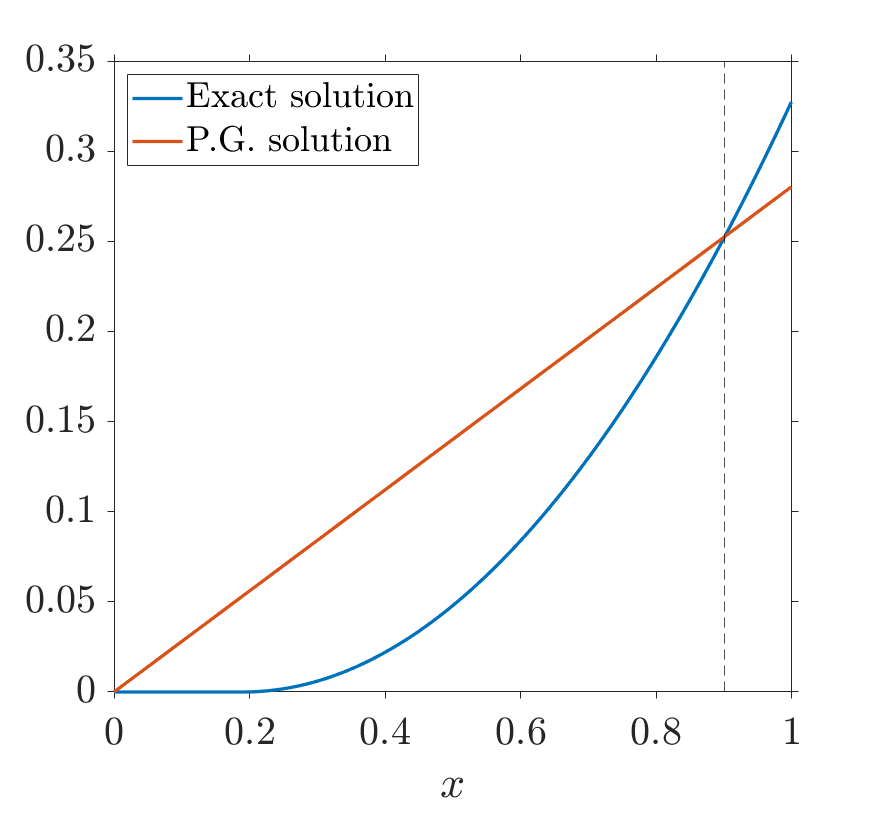}
    \caption{One element}
    \label{fig:1d_solutions_a}
  \end{subfigure}
  \hfill
  \begin{subfigure}[b]{0.32\textwidth}
    \includegraphics[width=\textwidth, height=\textwidth]{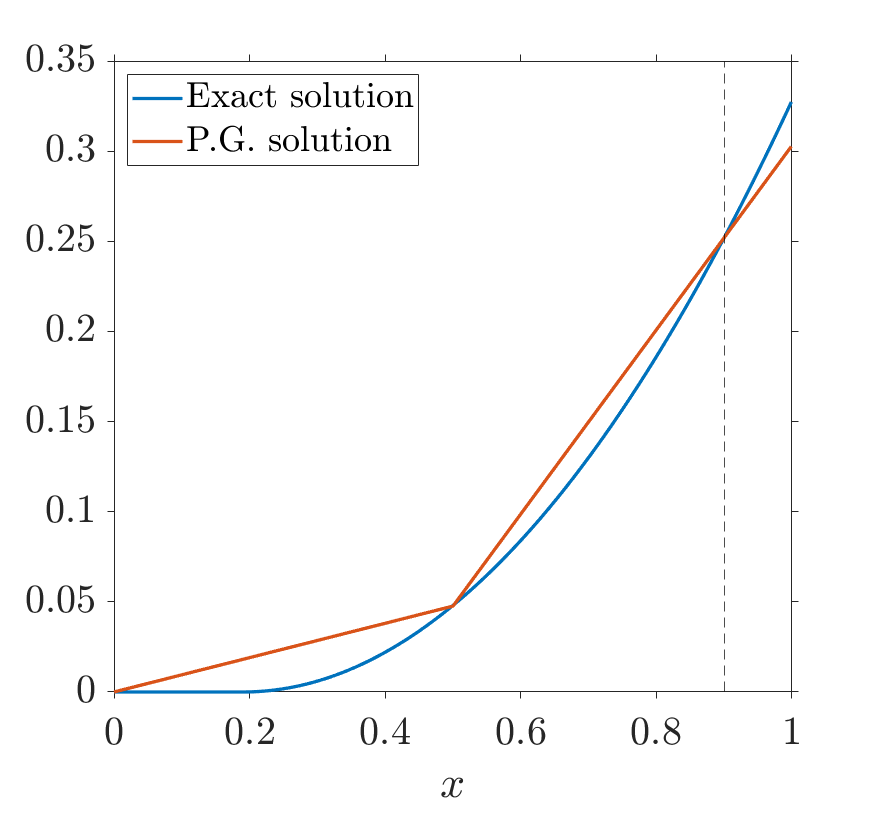}
    \caption{Two elements}
  \end{subfigure}
  \hfill
  \begin{subfigure}[b]{0.32\textwidth}
    \includegraphics[width=\textwidth, height=\textwidth]{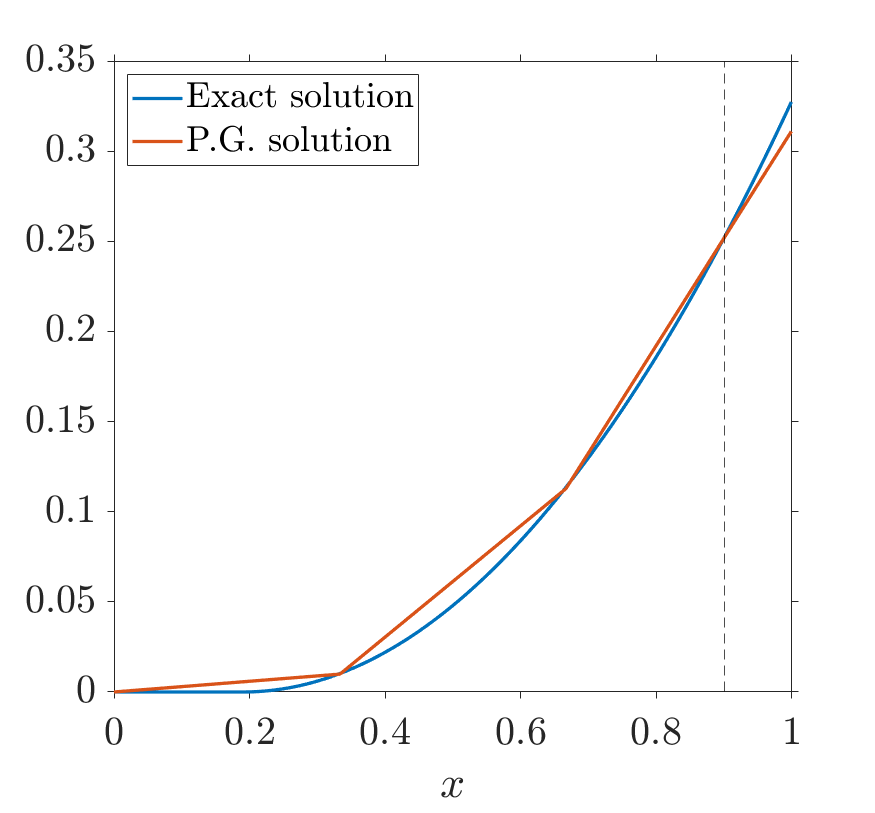}
    \caption{Three elements}
  \end{subfigure}
  \caption{Petrov-Galerkin solution with projected optimal test functions with trained weight. Dotted line shows the QoI location (0.9) and parameter value is $\lambda=0.19$.}
  \label{fig:1d_solutions}
\end{figure}

Figure~\ref{fig:1d_error} displays the QoI error $|q_{\xnot}(u_\lambda)-q_{\xnot}(u_{\lambda,h,\omega^*})|$ for different values of $\lambda\in[0,1]$. When the ANN-training stops at a cost functional smaller than \texttt{tol}$=9\cdot 10^{-7}$, the QoI error remains smaller than $10^{-3}$ for all $\lambda\in [0.1]$. 
In particular, Figure~\ref{fig:1d_error_a} shows that even in the simplest case of one-degree of freedom, it is possible to get reasonable approximations of the QoI for the entire range of~$\lambda$. 
\begin{figure}[!t]
\begin{subfigure}[b]{0.32\textwidth}
    \includegraphics[width=\textwidth, height=\textwidth]{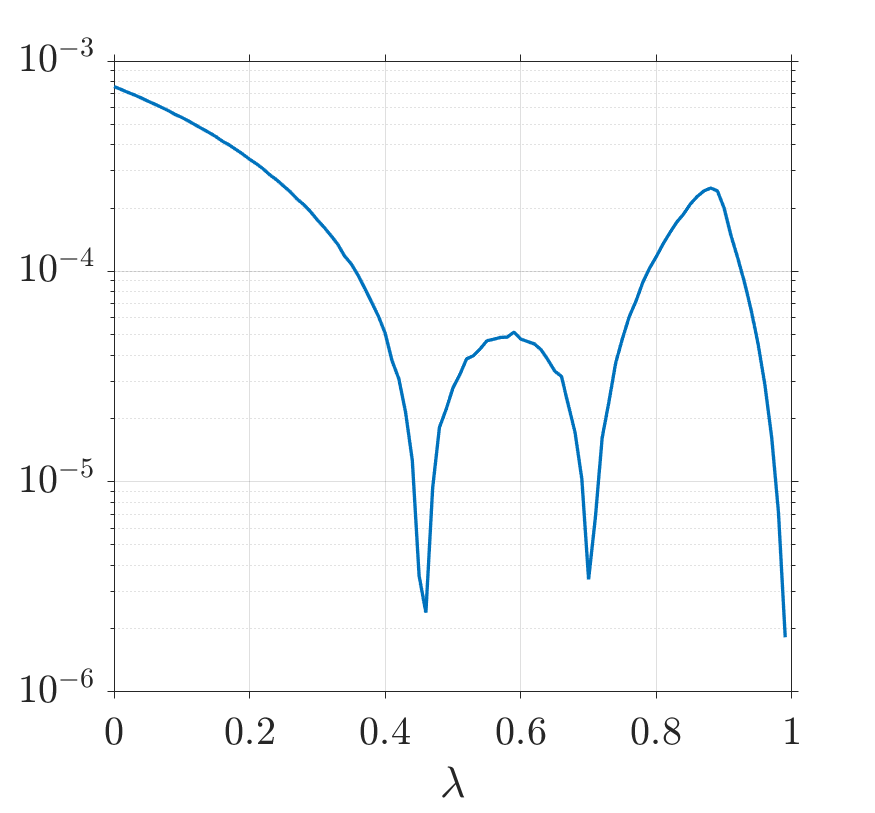}
    \caption{One DoF}
    \label{fig:1d_error_a}
  \end{subfigure}
  \hfill
  \begin{subfigure}[b]{0.32\textwidth}
    \includegraphics[width=\textwidth, height=\textwidth]{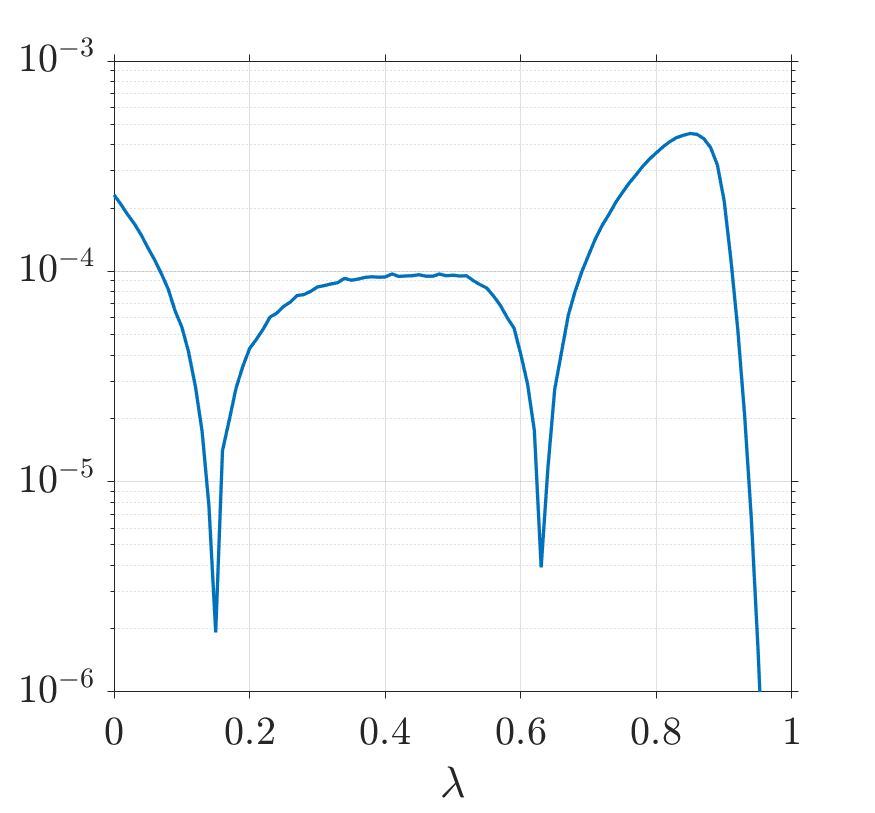}
    \caption{Two DoF}
  \end{subfigure}
  \hfill
  \begin{subfigure}[b]{0.32\textwidth}
    \includegraphics[width=\textwidth, height=\textwidth]{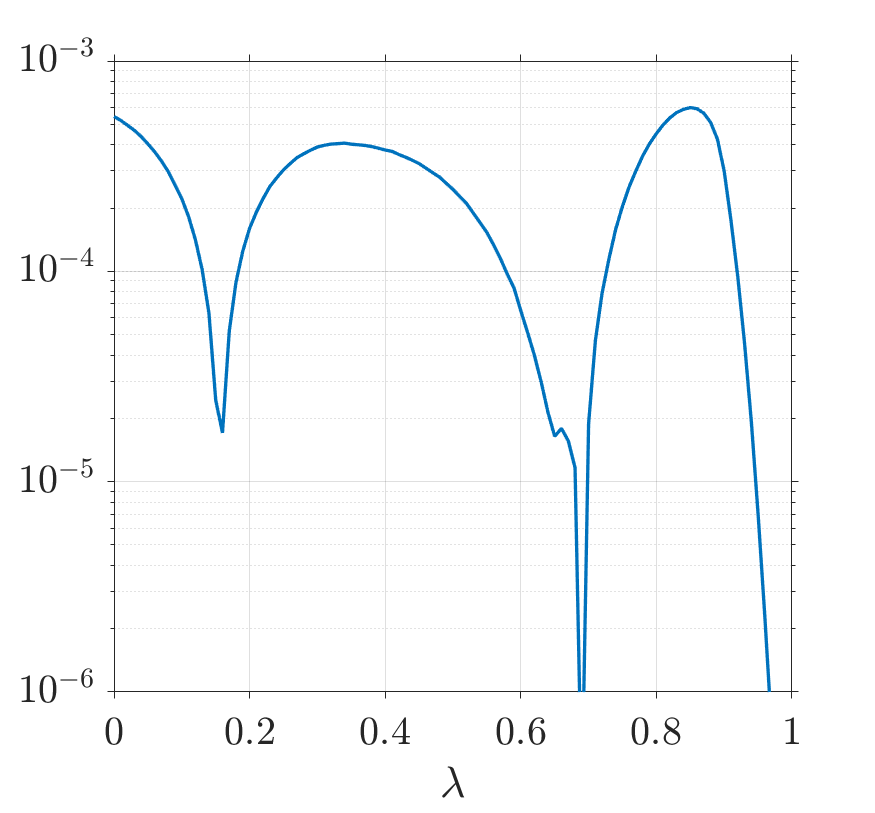}
    \caption{Three DoF}
  \end{subfigure}
  \caption{Absolute error between QoI of exact and approximate solutions for different $\lambda$ values.}
  \label{fig:1d_error}
\end{figure}

\subsection{1D advection with multiple QoIs}
This example is based on the same model problem of Section~\ref{sec:1D_adv}, but now we intend to approach two quantities of interest simultaneously:   
%
$q_{1}(u_\lambda) := u_\lambda(x_1)$ and $q_{2}(u_\lambda) := u_\lambda(x_2)$, 
where $x_1, x_2\in[0,1]$ are two different values. We also have considered now discrete trial spaces based on three, four and five elements. The training routine has been modified accordingly, and is driven now by the following minimization problem:
%
%
\begin{equation}
\nonumber
\left\{
\begin{array}{l}
\theta^*=\displaystyle \argmin_{\theta\in\Phi}  \displaystyle {1\over 2}\sum_{i=1}^{N_s}\left| q_1(u_{h,\lambda_i,\omega})-q_1(u_{\lambda_i})\right |^2 + \left| q_2(u_{h,\lambda_i,\omega})-q_2(u_{\lambda_i})\right |^2,\\
\mbox{subject to:} 
\left\{
\begin{array}{lll}
 \omega(\cdot) = \sigma(\ANN(\cdot\,;\theta)). & &\\
 (r_{h,\lambda_i,\omega},v_h)_{\mbbV,\omega} + b(u_{h,\lambda_i,\omega},v_h) & = \ell_{\lambda_i}(v_h), & \forall\, v_h\in\mbbV_h,\\
b(w_h,r_{h,\lambda_i,\omega}) & =0,  & \forall\, w_h\in \mbbU_h.
\end{array}
\right.
\end{array}
\right.
\end{equation}
For this example, we consider a training set of size $N_s=12$ with $\lambda_i = (i-1)/11$, for all $i=1,\dots,12$. The weight $\omega(x)$ will be described by the sigmoid of an artificial neural network  that depends on one single hidden layer and $N_n = 6$ hidden neurons.
Numerical results are depicted in Figures~\ref{fig:1d_solutions_2qoi} and~\ref{fig:2qoi_error} for $x_1=0.3$ and $x_2=0.7$. Accurate values of both QoIs are obtained for the entire range of~$\lambda$. These results are roughly independent of the size of the trial space.
%
%
\begin{figure}[!t]
\begin{subfigure}[b]{0.32\textwidth}
    \includegraphics[width=\textwidth, height=\textwidth]{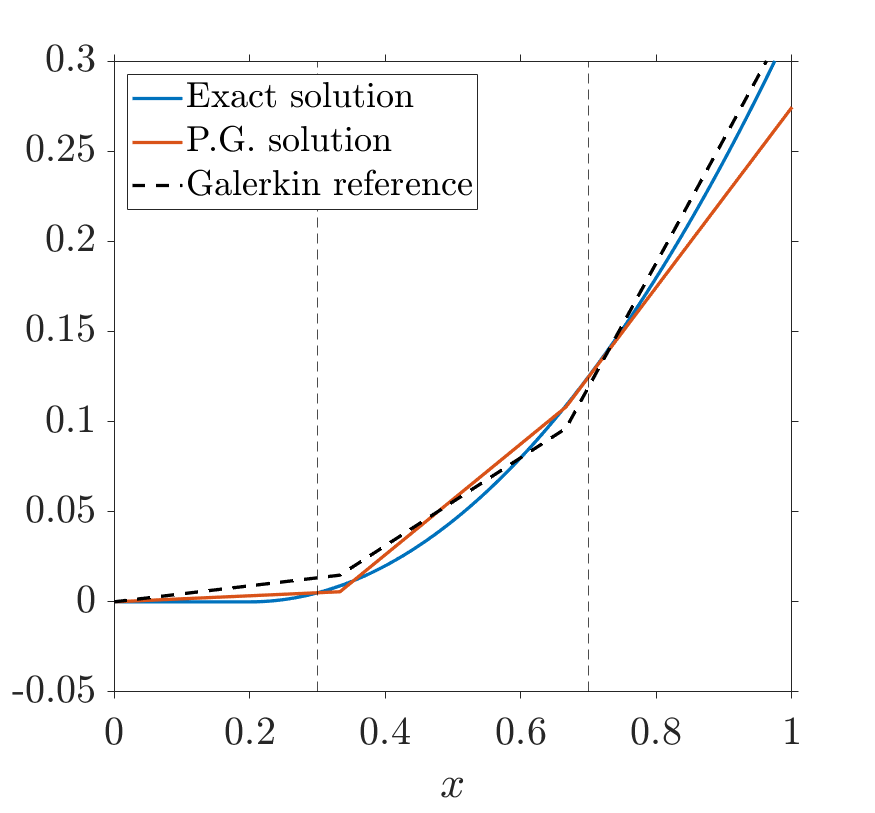}
    \caption{Three elements}
  \end{subfigure}
  \hfill
  \begin{subfigure}[b]{0.32\textwidth}
    \includegraphics[width=\textwidth, height=\textwidth]{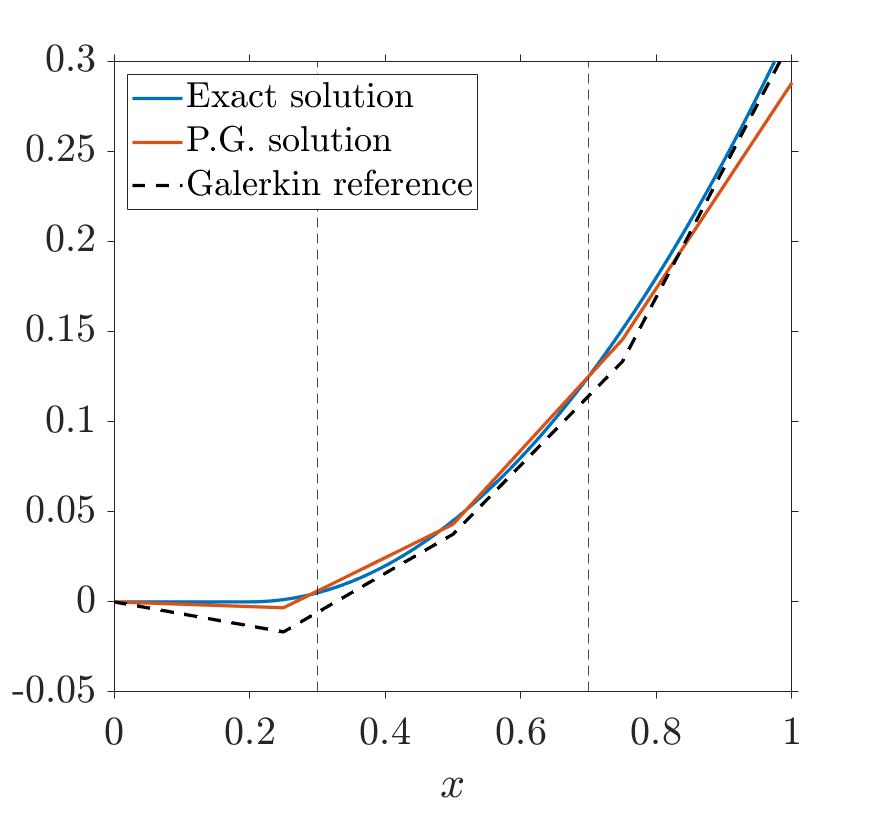}
    \caption{Four elements}
  \end{subfigure}
  \hfill
  \begin{subfigure}[b]{0.32\textwidth}
    \includegraphics[width=\textwidth, height=\textwidth]{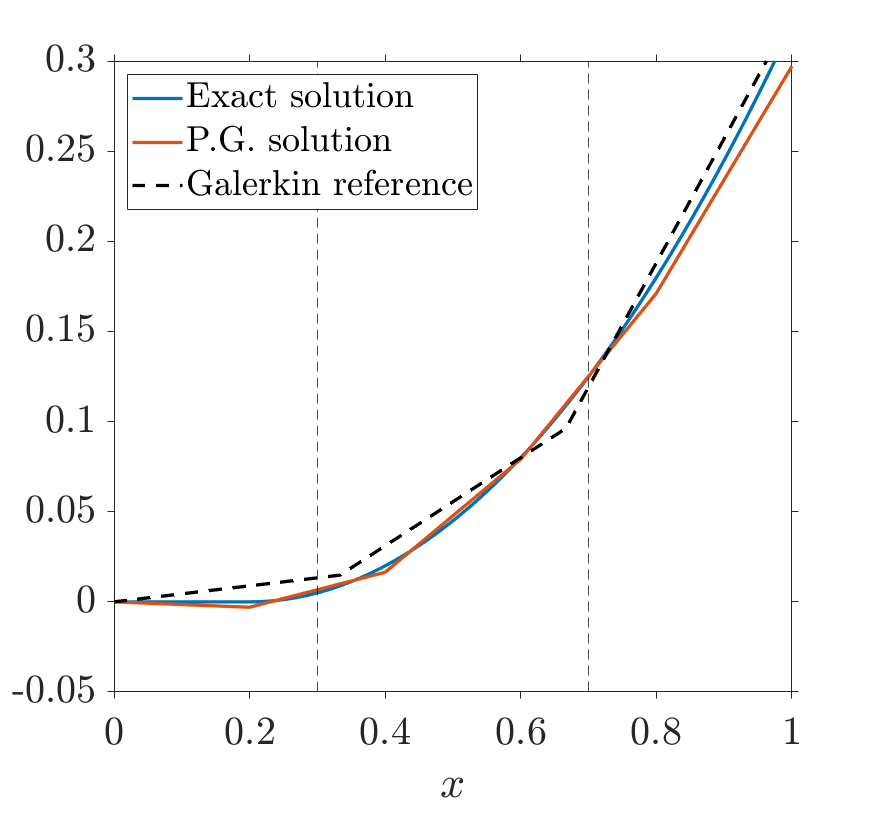}
    \caption{Five elements}
  \end{subfigure}
  \caption{Petrov-Galerkin solution with projected optimal test functions with trained weight. Dotted lines show the QoI locations (0.3 and 0.7) and parameter value is $\lambda=0.2$.}
  \label{fig:1d_solutions_2qoi}
\end{figure}

\begin{figure}[!t]
\begin{subfigure}[b]{0.32\textwidth}
    \includegraphics[width=\textwidth, height=\textwidth]{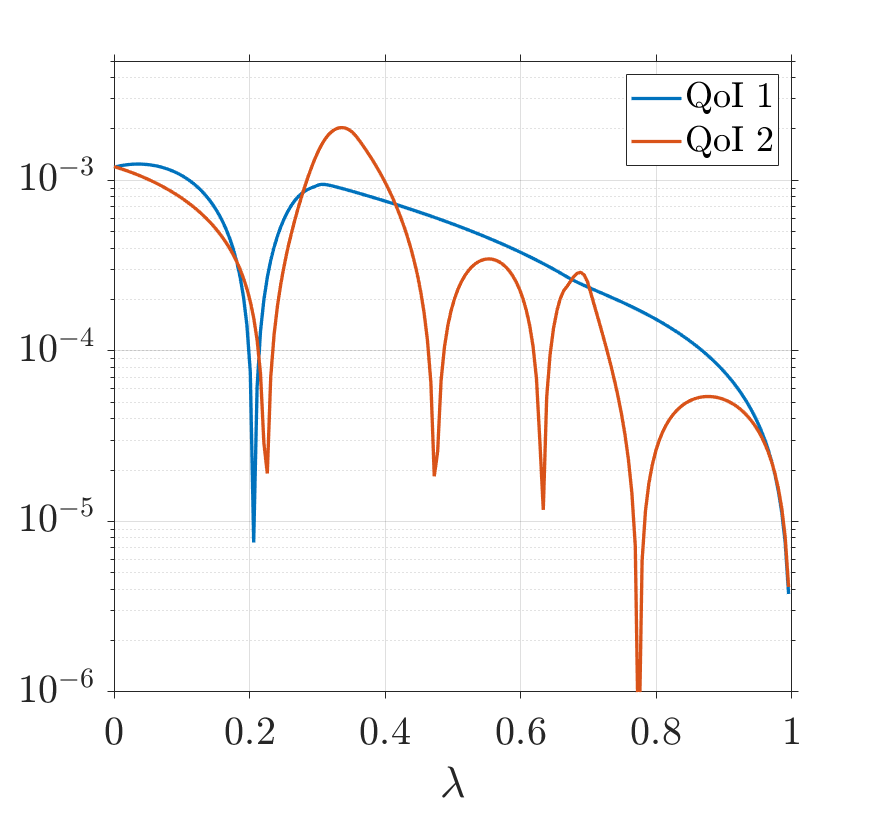}
    \caption{Three DoF}
  \end{subfigure}
  \hfill
  \begin{subfigure}[b]{0.32\textwidth}
    \includegraphics[width=\textwidth, height=\textwidth]{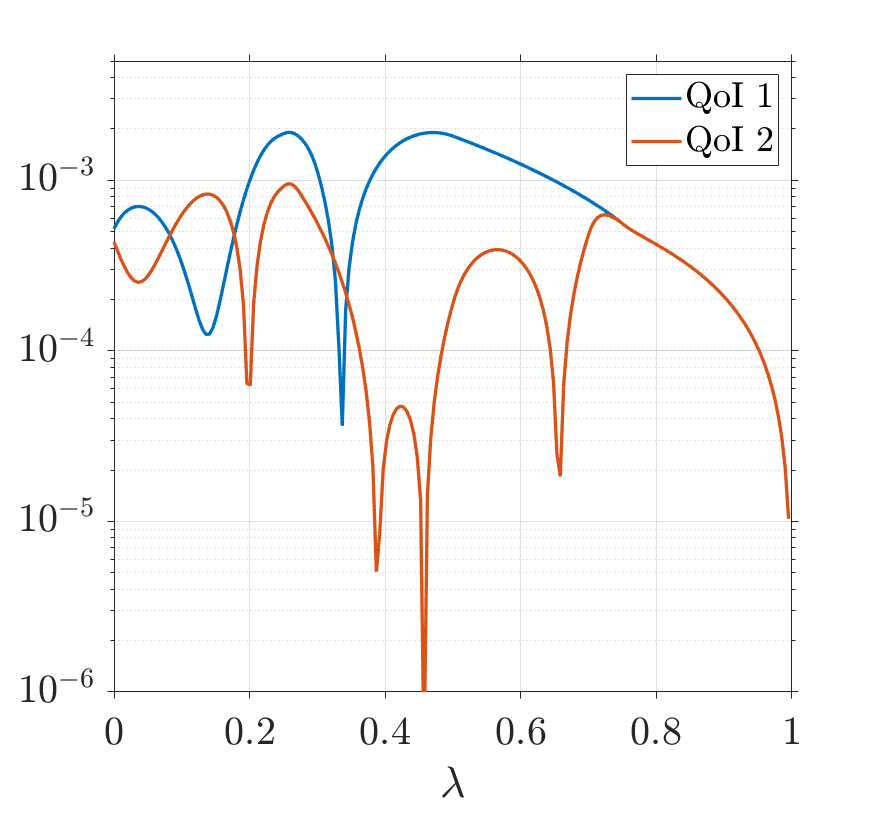}
    \caption{Four DoF}
  \end{subfigure}
  \hfill
  \begin{subfigure}[b]{0.32\textwidth}
    \includegraphics[width=\textwidth, height=\textwidth]{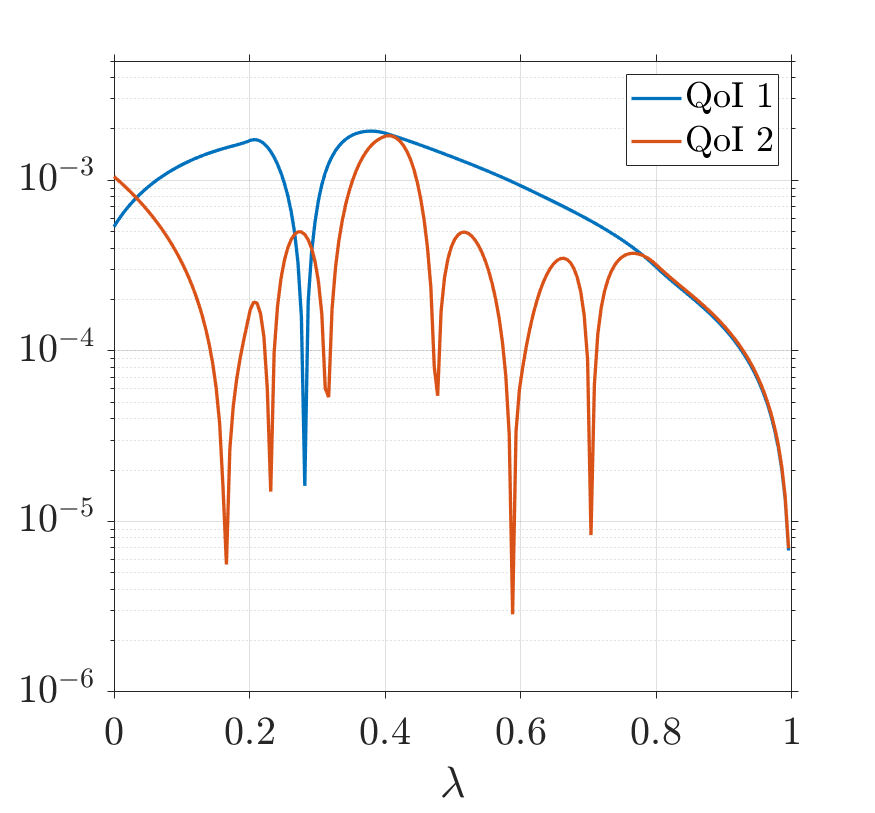}
    \caption{Five DoF}
  \end{subfigure}
  \caption{Absolute error between QoI of exact and approximate solutions for different $\lambda$ values, for each QoI $q_1(u) = u(0.3)$ and $q_2(u) = u(0.7)$.}
  \label{fig:2qoi_error}
\end{figure}

\subsection{2D diffusion with one QoI}
Consider the two-dimensional unit square $\Omega=[0,1]\times[0,1]$ and the family of PDEs:
\begin{equation}\label{eq:pde_example1}
\left\{
\begin{array}{r@{}ll}
-\Delta u&{}=f_{\lambda} & \hbox{in } \Omega, \\
u &{}=0 & \hbox{over } \partial\Omega,
\end{array}
\right.
\end{equation}
where the family of functions $\{f_{\lambda}\}_{\lambda\in(0,1)}$ is described by the formula:
\begin{equation}
\nonumber
\begin{split}
f_{\lambda}(x_1,x_2) = 2\pi^2(1+\lambda^2)\sin(\pi x_1)\sin(\lambda \pi x_1)\sin(\pi x_2)\sin(\lambda \pi x_2) - \\ 2\lambda\pi^2[ \cos(\pi x_1)\cos(\lambda \pi x_1)\sin(\pi x_2)\sin(\lambda \pi x_2)+ \\ \sin(\pi x_1)\sin(\lambda \pi x_1)\cos(\pi x_2)\cos(\lambda \pi x_2)].
\end{split}
\end{equation}
Accordingly, the reference exact solution of~\eqref{eq:pde_example1} is:
$$
u_{\lambda}(x) = \sin(\pi x_1)\sin(\lambda \pi x_1)\sin(\pi x_2)\sin(\lambda \pi x_2).
$$
The quantity of interest chosen for this example will be the average 
\begin{equation}\label{eq:QoI_ex1_2D}
q(u_\lambda) := \frac{1}{|\Omega_0|}\int_{\Omega_0} u_\lambda(x)\,dx,
\end{equation}
with $\Omega_0 := [0.79\, ,\, 0.81]\times [0.39\, ,\, 0.41]\subset\Omega$ (see~Figure~\ref{fig:2d_meshes}).

The variational formulation of problem~\eqref{eq:pde_example1} will be:
\begin{equation}
\nonumber
\left\{
\begin{array}{l}
\text{Find } u_\lambda\in \mbbU \text{ such that:} \\
b(u_\lambda,v):=\displaystyle\int_{\Omega} \nabla u_\lambda\cdot \nabla v = \int_{\Omega} f_{\lambda}v=:\ell_{\lambda}(v),\qquad\forall v\in \mbbV, 
\end{array}
\right.
\end{equation}
where $\mbbU = \mbbV = H^1_{0}(\Omega)$, and
$\mbbV$ is endowed with the weighted inner-product:
$$ (v_1,v_2)_{\mathbb V,\omega}:=\int_{\Omega} \omega\, \nabla v_1\cdot \nabla v_2\,,\quad
\forall v_1,v_2\in \mbbV.$$
As in the previous example, the weight is going to be determined using an artificial neural network so that
$\omega(x_1,x_2)=\sigma(\ANN(x_1,x_2;\theta))$. Such a network is composed by one single hidden layer, and $N_n = 5$ hidden neurons. Hence, $\theta$ contains 20 parameters to estimate, i.e.,  
$$
\mathrm{ANN}(x_1,x_2;\theta) = \sum_{j=1}^{N_n} \theta_{j4} \sigma(\theta_{j1}x_1 + \theta_{j2}x_2 + \theta_{j3}).
$$
To train the $\ANN$, we use the inputs $\{\lambda_i\}_{i=1}^{9}$, where $\lambda_i = 0.125(i-1)$, and its corresponding quantities of interest $\{q(u_{\lambda_i})\}_{i=1}^{9}$, by means of equation~\eqref{eq:QoI_ex1_2D}. Again, the training procedure is based on the constrained minimization~\eqref{eq:eqmin_prob_mixed}. For the experiments, we use coarse discrete trial spaces $\mbbU_h$ having one, five, and eight degrees of freedom respectively (see~Figure~\ref{fig:2d_meshes}). In each case,  the test space $\mbbV_h$ has been set to be a piecewise quadratric conforming polynomial space, over a uniform triangular mesh of 1024 elements.
The minimization algorithm~\eqref{eq:eqmin_prob_mixed} stops when a tolerance \texttt{tol}$=9\cdot 10^{-7}$ is reached.

The errors on the QoI are depicted in Figure~\ref{fig:2d_error} for each trial space under consideration, and show relative errors below~$10^{-3}$ for the entire range of~$\lambda$.  
%
%
\begin{figure}[!t]
\begin{subfigure}[b]{0.32\textwidth}
    \includegraphics[width=\textwidth, height=\textwidth]{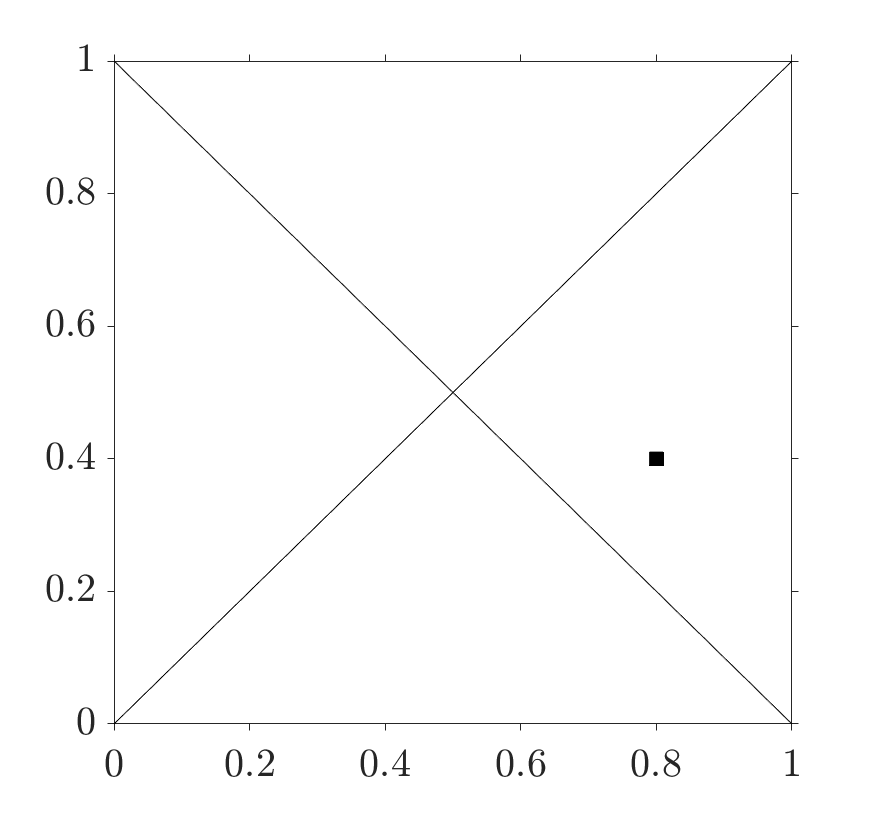}
    \caption{One DoF}
  \end{subfigure}
  \hfill
  \begin{subfigure}[b]{0.32\textwidth}
    \includegraphics[width=\textwidth, height=\textwidth]{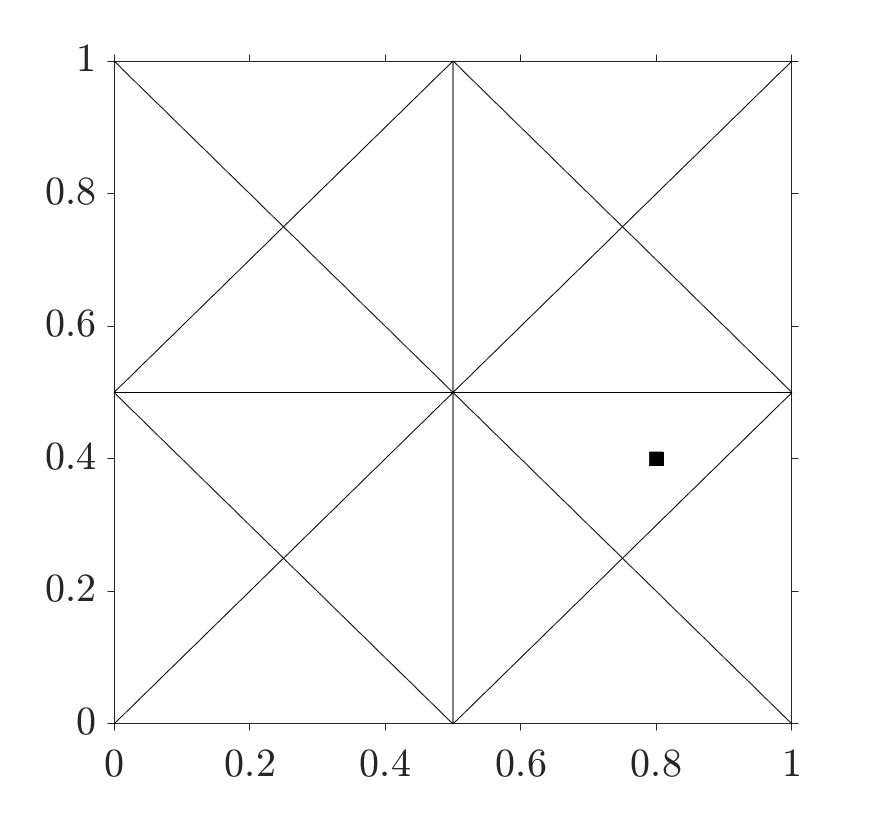}
    \caption{Five DoF}
  \end{subfigure}
  \hfill
  \begin{subfigure}[b]{0.32\textwidth}
    \includegraphics[width=\textwidth, height=\textwidth]{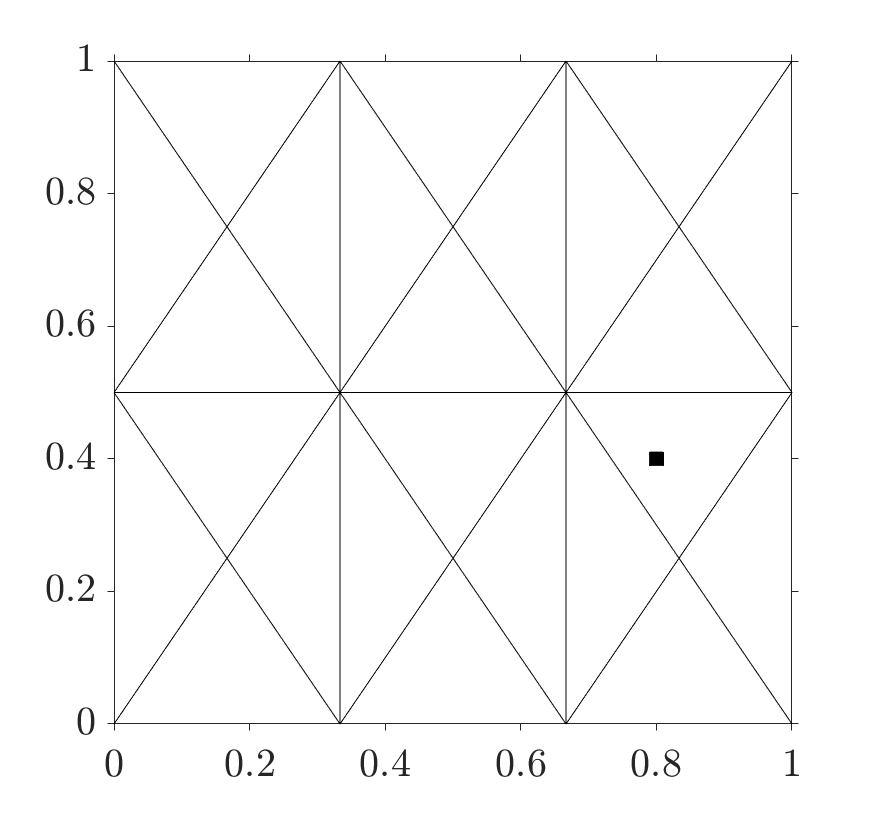}
    \caption{Eight DoF}
  \end{subfigure}
  \caption{Meshes considered for the discrete trial space $\mbbU_h$. The black square represent the quantity of interest location $\Omega_0=[0.79\, ,\, 0.81]\times [0.39\, ,\, 0.41]$.}
  \label{fig:2d_meshes}
\end{figure}
\begin{figure}[!t]
\begin{subfigure}[b]{0.32\textwidth}
    \includegraphics[width=\textwidth, height=\textwidth]{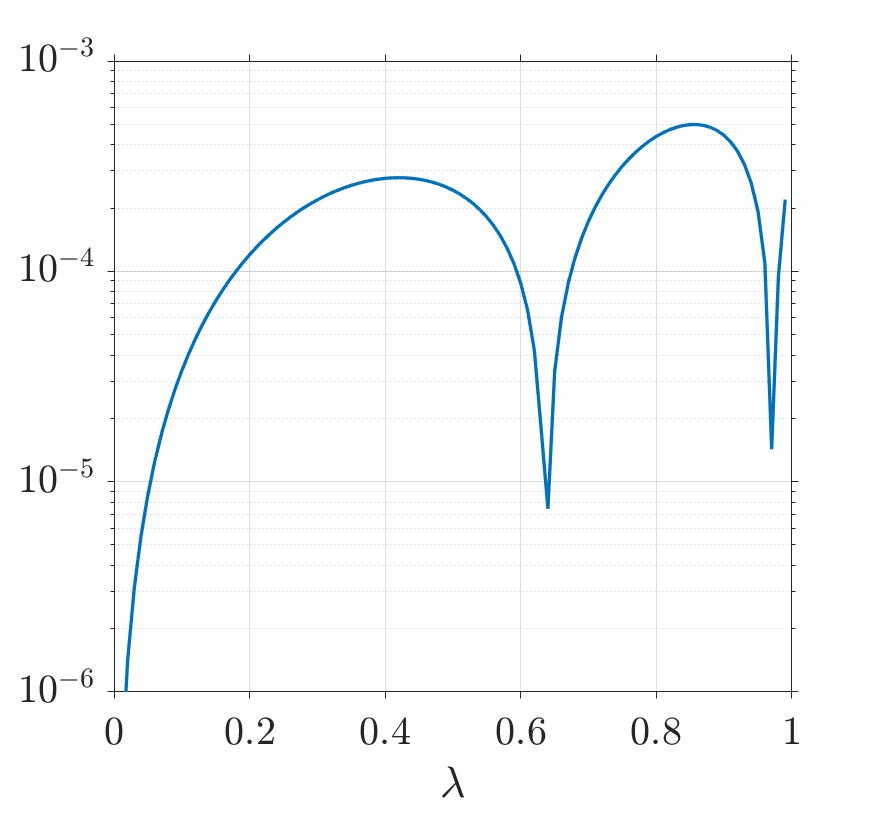}
    \caption{One DoF}
    \label{fig:2d_error_a}
  \end{subfigure}
  \hfill
  \begin{subfigure}[b]{0.32\textwidth}
    \includegraphics[width=\textwidth, height=\textwidth]{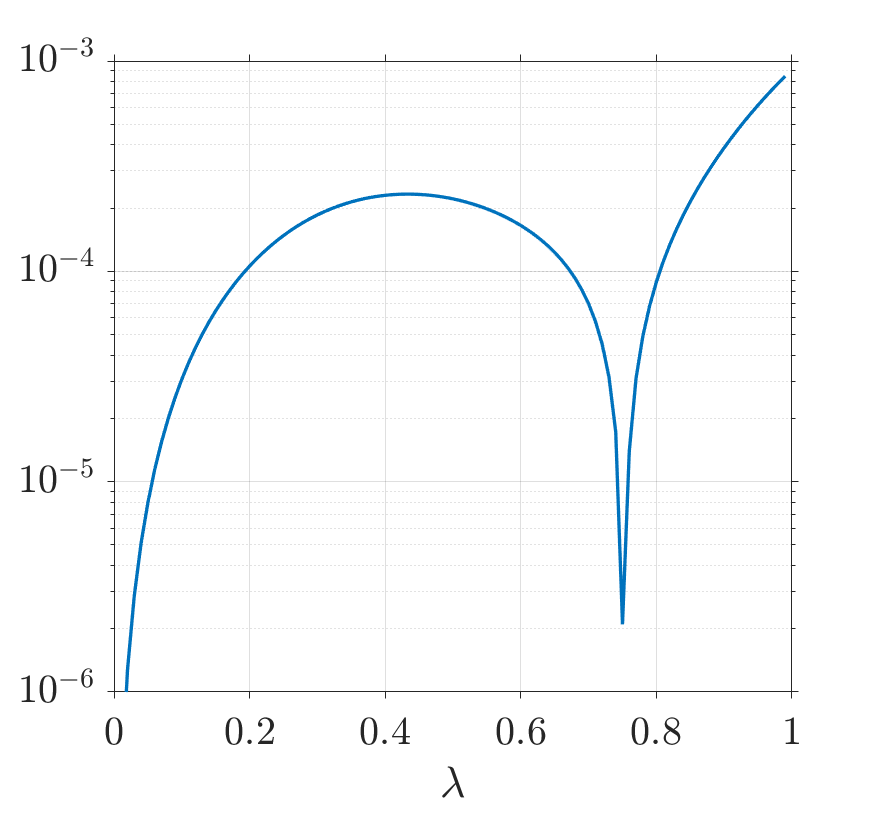}
    \caption{Five DoF}
  \end{subfigure}
  \hfill
  \begin{subfigure}[b]{0.32\textwidth}
    \includegraphics[width=\textwidth, height=\textwidth]{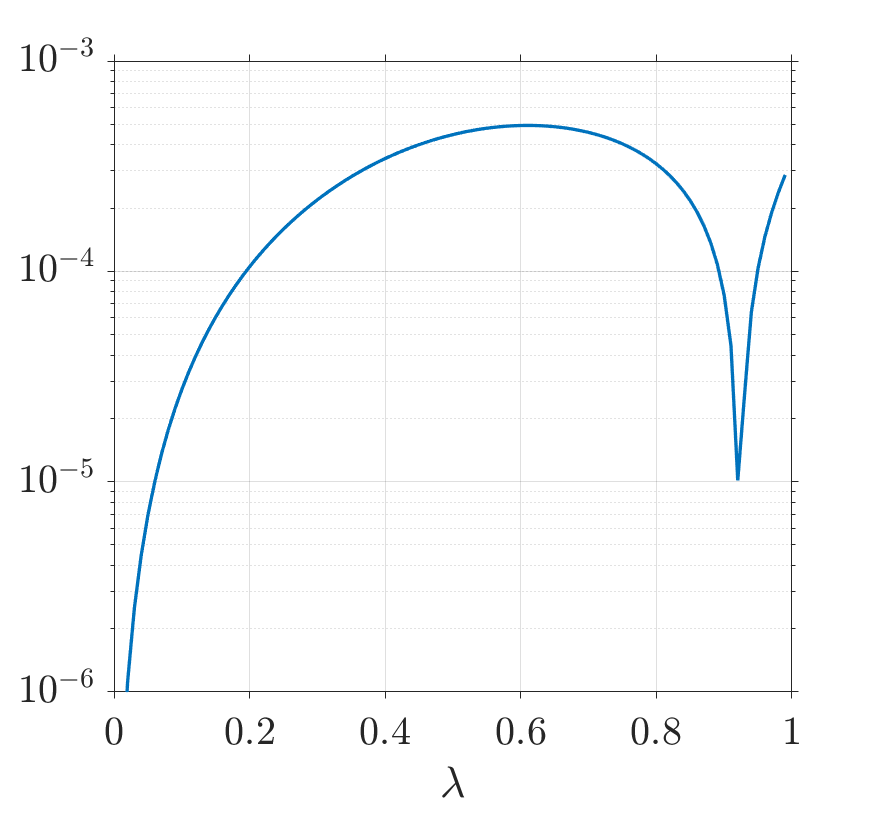}
    \caption{Eight DoF}
  \end{subfigure}
  \caption{Absolute error between QoI of exact and approximate solutions for different $\lambda$ values.}
  \label{fig:2d_error}
\end{figure}

%

\section{Conclusions}
\label{sec:concl}
In this paper, we introduced the concept of data-driven finite element methods. These methods are tuned within a machine-learning framework, and are tailored for the accurate computation of output quantities of interest, regardless of the underlying mesh size. In fact, on coarse meshes, they deliver significant improvements in the accuracy compared to standard methods on such meshes.
\par
We presented a stability analysis for the discrete method, which can be regarded as a Petrov--Galerkin scheme with parametric test space that is equivalent to a minimal-residual formulation measuring the residual in a (discrete) dual weighted norm. Numerical examples were presented for elementary one- and two-dimensional elliptic and hyperbolic problems, in which weight functions are tuned that are represented by artificial neural networks with up to 20~parameters.
\par
Various extensions of our methodology are possible. While we only focussed on linear quantities of interest, nonlinear ones can be directly considered. Also, it is possible to consider a dependence of the bilinear form on~$\lambda$, however, this deserves a completely separate treatment, because of the implied $\lambda$-dependence of the trial-to-test map and the $B$-matrix. An open problem of significant interest is the dependence of the performance of the trained method on the richness of the parametrized weight function. While we showed that in the simplest example of 1-D diffusion with one degree of freedom, the weight function allows for exact approximation of quantities of interest, it is not at all clear if this is valid in more general cases, and what the effect is of (the size of) parametrization. 
%
\appendix
\label{sec:appendix}
\section{Proof of Theorem~\ref{thm:discrete_wellposedness}}\label{sec:proof}
The mixed scheme~\eqref{eq:mixed_system} has a classical \emph{saddle point} structure, which is uniquely solvable since the \emph{top left} bilinear form is an inner-product (therefore coercive) and $b(\cdot,\cdot)$ satisfies the discrete inf-sup condition~\eqref{eq:inf-sup} {(see, e.g.,~\cite[Proposition~2.42]{ErnGueBOOK})}.  

The a~priori estimates~\eqref{eq:apriori_estimates} are very well-known in the residual minimization FEM literature (see, e.g.,~\cite{GopQiuMOC2014, BroSteCAMWA2014, MugZeeARXIV2018}). However, for the sake of completeness, we show here how to obtain them. Let $v_{h,\lambda,\omega}\in\mbbV_h$ be such that (cf.~\eqref{eq:T_h,omega}):
\begin{equation}\label{eq:supremizer}
(v_{h,\lambda,\omega},v_h)_{\mbbV,\omega}=b(u_{h,\lambda,\omega},v_h),\quad\forall 
v_h\in\mbbV_h\,.
\end{equation}
In particular, combining eq.~\eqref{eq:supremizer} with eq.~\eqref{eq:mixed_system_b} of the mixed scheme, we get the orthogonality property:
\begin{equation}\label{eq:r_orthogonal}
(v_{h,\lambda,\omega},r_{h,\lambda,\omega})_{\mbbV,\omega}=b(u_{h,\lambda,\omega},r_{h,\lambda,\omega})=0\,.
\end{equation}
To get the first estimate observe that:
\begin{alignat}{2}
\|u_{h,\lambda,\omega}\|_\mbbU  \leq  & {1\over\gamma_{h,\omega}}
\sup_{v_h\in\mbbV_h}{|b(u_{h,\lambda,\omega},v_h)|\over \|v_h\|_{\mbbV,\omega}}\tag{by~\eqref{eq:inf-sup}}\\
= & {1\over\gamma_{h,\omega}}
\sup_{v_h\in\mbbV_h}{|(v_{h,\lambda,\omega},v_h)_{\mbbV,\omega}|\over \|v_h\|_{\mbbV,\omega}}
\tag{by~\eqref{eq:supremizer}}\\
= & {1\over\gamma_{h,\omega}}
{|(v_{h,\lambda,\omega},v_{h,\lambda,\omega})_{\mbbV,\omega}|\over \|v_{h,\lambda,\omega}\|_{\mbbV,\omega}}\tag{since $v_{h,\lambda,\omega}$ is the supremizer}\\
= & {1\over\gamma_{h,\omega}}
{|(r_{h,\lambda,\omega}+v_{h,\lambda,\omega},v_{h,\lambda,\omega})_{\mbbV,\omega}|\over \|v_{h,\lambda,\omega}\|_{\mbbV,\omega}}\tag{by~\eqref{eq:r_orthogonal}}\\
= & {1\over\gamma_{h,\omega}}
{|\ell_\lambda(v_{h,\lambda,\omega})|\over \|v_{h,\lambda,\omega}\|_{\mbbV,\omega}}
\tag{by~\eqref{eq:supremizer}~and~\eqref{eq:mixed_system_a}}\\
= & {1\over\gamma_{h,\omega}}
{|b(u_\lambda,v_{h,\lambda,\omega})|\over \|v_{h,\lambda,\omega}\|_{\mbbV,\omega}}
\tag{by~\eqref{eq:VarFor_gen_prob}}\\
\leq & {M_\omega\over\gamma_{h,\omega}}\|u_\lambda\|_\mbbU\,.
\tag{by~\eqref{eq:B_constants}}
\end{alignat}

For the second estimate we define the projector $P:\mbbU\to\mbbU_h$, such that $Pu\in\mbbU_h$ corresponds to the second component of the solution of the mixed system~\eqref{eq:mixed_system} with right hand side $b(u,\cdot)\in\mbbV^*$ in~\eqref{eq:mixed_system_a}. We easily check that $P$ is a bounded linear proyector satisfying
$P^2=P\neq 0,I$, and $\|P\|\leq M_\omega/\gamma_{h,\omega}$. Hence, from Kato's identity $\|I-P\|=\|P\|$ for Hilbert space projectors~\cite{KatNM1960}, we get for any $w_n\in\mbbU_n$:
\begin{equation}\notag
\|u_\lambda-u_{h,\lambda,\omega}\|_\mbbU=\|(I-P)u_\lambda\|_\mbbU=\|(I-P)(u_\lambda-w_h)\|_\mbbU\leq \|P\|\|u_\lambda-w_h\|_\mbbU\,.
\end{equation}
Thus, the a~priori error estimate follows using the bound of $\|P\|$ and taking the infimum over all $w_h\in\mbbU_h$.

%
\bibliography{BibFile}
\bibliographystyle{siam}
%
%
%
%
\end{document}